%% file: supercritical_standing_rings.tex
\documentclass[preprint,11pt]{elsarticle}

\usepackage{subfig,xcolor,graphicx}

\usepackage{fullpage}

\usepackage{amsmath,amssymb,amsthm}

\newtheorem{thrm}{Theorem}
\newtheorem{lem}[thrm]{Lemma}

\newtheorem{conj}[thrm]{Conjecture}
\newtheorem{cor}[thrm]{Corollary}
\newtheorem{rmk}{Remark}[section]

\newcommand{\LC}{\kappa}
\newcommand{\LCB}{\kappa_{\rm B}}

\newcommand{\rmax}{ { {\it r_{\max}}  } }

\DeclareMathOperator{\sech}{sech}

\newcommand{\mycaption}[1]{\parbox{0.95\textwidth}{\caption{{#1}}}}

\newcommand{\abs}[1]{\left\vert{#1}\right\vert}
\newcommand{\norm}[1]{\left\Vert{#1}\right\Vert}

\graphicspath{{}}

\begin{document}
\begin{frontmatter}
\title{Singular standing-ring solutions of nonlinear partial differential equations}
\author[tau]{Guy Baruch}\ead{guy.baruch@math.tau.ac.il}
\author[tau]{Gadi Fibich\corref{cor1}}\ead{fibich@tau.ac.il}
\author[tau]{Nir Gavish}\ead{nirgvsh@tau.ac.il}
\address[tau]{School of Mathematical Sciences, Tel Aviv University, Tel Aviv 69978, Israel}
\cortext[cor1]{Corresponding author}
\begin{abstract}
We present a general framework for constructing singular solutions
of nonlinear evolution equations that become singular on
a~$d$-dimensional sphere, where $d>1$. The asymptotic profile and
blowup rate of these solutions are the same as those of solutions of
the corresponding one-dimensional equation that become singular at a
point. We provide a detailed numerical investigation of these new
singular solutions for the following equations: The nonlinear
Schr\"odinger equation~$i\psi_ t( t,{\bf
x})+\Delta\psi+|\psi|^{2\sigma}\psi=0$ with~$\sigma>2$, the
biharmonic nonlinear Schr\"odinger equation~$i\psi_t( t,{\bf
x})-\Delta^2\psi+|\psi|^{2\sigma}\psi=0$ with~$\sigma>4$, the
nonlinear heat equation~$\psi_t( t,{\bf
x})-\Delta\psi-|\psi|^{2\sigma}\psi=0$ with~$\sigma>0$, and the
nonlinear biharmonic heat equation~$\psi_t( t,{\bf
x})+\Delta^2\psi-|\psi|^{2\sigma}\psi=0$ with~$\sigma>0$.
\end{abstract}
\end{frontmatter}

\section{\label{sec:intro}Introduction}

In this study, we consider nonlinear evolution
equations of the form
\begin{equation}
\label{eq:NL_PDE_form}
u_t(t,{\bf x})=F[u,\Delta u,\Delta^2u,\cdots], \qquad {\bf x}\in\mathbb{R}^d, \quad d>1.
\end{equation}
Examples for such equations are the nonlinear Schr\"odinger equation,
the biharmonic nonlinear Schr\"odinger equation,
the nonlinear heat equation,
and the biharmonic  nonlinear heat equation.
It is well known that these equations admit solutions that become singular
at a point. Recently, it was discovered that the nonlinear Schr\"odinger equation with a quintic nonlinearity
admits solutions that become singular on
a $d$-dimensional sphere~\cite{Raphael-06,SC_rings-07,Vortex_rings-08,Raphael-08},
see Figure~\ref{fig:standingRingIlustration}. Following~\cite{SC_rings-07},
we refer to these solutions as {\em singular standing-ring} solutions.

The main goal of this study is to present a general framework for
constructing singular standing-ring solutions of nonlinear evolution
equations of the form~(\ref{eq:NL_PDE_form}). In order to understand
the basic idea, let us assume that equation~\eqref{eq:NL_PDE_form}
admits a singular standing-ring solution. Then, near the
singularity, equation~\eqref{eq:NL_PDE_form} reduces to the
one-dimensional equation
\begin{equation}
\label{eq:NL_PDE_form_one_dimensional}
u_t(t,r)=F[u,u_{rr},u_{rrrr},\cdots], \qquad r = |{\bf x}|.
\end{equation}
Hence, equation~(\ref{eq:NL_PDE_form_one_dimensional}) ``should''
admit a solution that becomes singular at a point.  Conversely, if
the one-dimensional equation~(\ref{eq:NL_PDE_form_one_dimensional})
admits a solution that becomes singular at a point, then
equation~\eqref{eq:NL_PDE_form} ``should'' admit a standing-ring
singular solution. Moreover, the asymptotic profile and blowup rate
of the standing-ring solutions of~\eqref{eq:NL_PDE_form}
 ``should'' be the same as those of the corresponding solution of the one-dimensional
equation~(\ref{eq:NL_PDE_form_one_dimensional}).

The above argument is obviously very informal. Nevertheless, in what
follows we will provide numerical evidence in support of the
relation between standing-ring singular solutions
of~\eqref{eq:NL_PDE_form} and singular solutions of the
one-dimensional equation~(\ref{eq:NL_PDE_form_one_dimensional}).

\begin{figure}
    \begin{center}
    \scalebox{.85}{\includegraphics{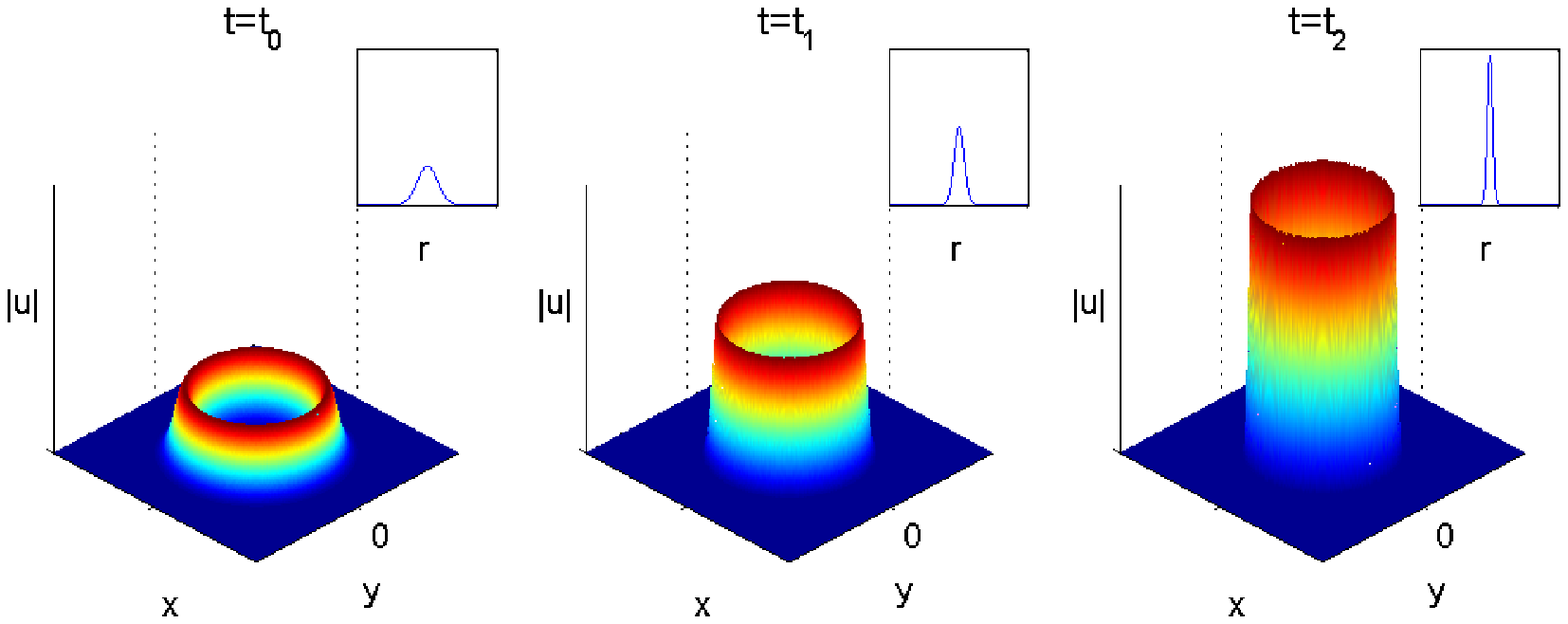}}
    \mycaption{Illustration of a two-dimensional singular standing-ring solution at
    times~$t_0<t_1<t_2$.
    The insets show the radial profile of~$u$.
     \label{fig:standingRingIlustration}
     }
    \end{center}
\end{figure}
\subsection{Peak-type and ring-type singular solutions of the nonlinear Schr\"odinger equation (NLS) - review}
The focusing nonlinear Schr\"odinger equation (NLS)
\begin{equation} \label{eq:intro_NLS}
    i\psi_ t( t,{\bf x})+\Delta\psi+|\psi|^{2\sigma}\psi=0,\qquad
    \psi(0,{\bf x})=\psi_0({\bf x}),
\end{equation}
where ${\bf x}\in\mathbb{R}^d$ and
$\Delta=\partial_{x_1x_1}+\cdots+\partial_{x_dx_d}$, is one of the
canonical nonlinear equations in physics, arising in various fields
such as nonlinear optics, plasma physics, Bose-Einstein condensates,
and surface waves.  One of the important properties of the NLS is
that it admits solutions which become singular at a finite time,
i.e.,
\[ \lim_{ t\to T_c}\|\psi\|_{H^1}=\infty,\qquad 0< T_c<\infty. \]
The NLS is called {\em subcritical} if $\sigma d<2$. In this case,
all solutions exist globally. In contrast, solutions of the {\em
critical}~($\sigma d=2$) and {\em supercritical}~($\sigma d>2$) NLS
can become singular in finite time.

When the initial condition~$\psi_0$ is radially-symmetric,
equation~\eqref{eq:intro_NLS} reduces to
\begin{equation}   \label{eq:radial-NLS}
   i\psi_t(t,r) + \psi_{rr}+\frac{d-1}{r}\psi_r +|\psi|^{2\sigma}\psi = 0,
   \qquad \psi(0,r)=\psi_0(r),\qquad d>1,
\end{equation}
where $r=|{\bf x}|$. Let us denote the location of the maximal
amplitude by
\[
\rmax(t) = \arg \displaystyle\max_r|\psi|.
\]
Singular solutions of~\eqref{eq:radial-NLS} are called `{\em
peak-type}' when~$\rmax(t)\equiv0$ for~$0\le t \le T_c$, and `{\em
ring-type}' when~$\rmax(t)>0$ for~$0\le t < T_c$.

Until a few years ago, the only known singular NLS solutions were
peak-type.  In the critical case~$\sigma d=2$, it has been
rigorously shown~\cite{Merle-03} that peak-type solutions are
self-similar near the singularity, i.e.,~$\psi\sim\psi_R$, where
\begin{equation*}
\psi_R(t,r)=\frac{1}{L^{1/\sigma}(t)}R\left(\rho\right)e^{i\tau+i\frac{L_t}{4L}r^2},
\end{equation*}
\begin{equation*}
\tau=\int_0^t\frac{ds}{L^2(s)},\qquad\rho=\frac{r-r_{max}( t)}{L(t)},\qquad
r_{max}( t)\equiv0,
\end{equation*}
the self-similar profile~$R(\rho)$ attains its global maximum
at~$\rho=0$, and the blowup rate~$L(t)$ is given by the {\em loglog
law}
\begin{equation}L( t)\sim\left(\frac{2\pi( T_c- t)}{\log\log1/( T_c- t)}\right)^\frac{1}{2},\qquad t\to T_c.
\label{eq:logloglaw}
\end{equation}
In the supercritical case~$(\sigma d>2)$, the rigorous theory is far
less developed. However, formal calculations and numerical
simulations~\cite{Sulem-99} suggest that peak-type solutions of the
supercritical NLS collapse with the self-similar~$\psi_S$ profile,
i.e.,~$\psi\sim\psi_S$, where
\begin{subequations}\label{eq:intro_psiS}
\begin{equation}\label{eq:intro_psiS_profile}
\psi_S(t,r)=\frac{1}{L^{1/\sigma}(t)}S\left(\rho\right)e^{i\tau+i\frac{L_t}{4L}r^2},
\end{equation}
\begin{equation}\label{eq:intro_psiS_profile2}
\tau=\int_0^t\frac{ds}{L^2(s)},\qquad\rho=\frac{r-r_{max}( t)}{L(t)},\qquad
r_{max}( t)\equiv0,
\end{equation}
$|S(\rho)|$ attains its global maximum at~$\rho=0$, and the blowup
rate is a square-root, i.e.,
\begin{equation}\label{eq:intro_psiS_blowuprate}
L(t)\sim \kappa\sqrt{T_c-t},\qquad t\to T_c.
\end{equation}
\end{subequations}

 In the last few years, new singular solutions of the NLS
were discovered, which are
ring-type~\cite{Gprofile-05,SC_rings-07,Raphael-06,Raphael-08,Vortex_rings-08}.
In~\cite{SC_rings-07}, we showed that the NLS with~$d>1$
and~$\frac2d\le\sigma\le2$ admits singular ring-type solutions that
collapse with the~$\psi_Q$ profile, i.e.,~$\psi\sim\psi_Q$, where
\begin{subequations}
\label{eq:psiQ}
\begin{equation}
\psi_Q(t,r)=\frac{1}{L^{1/\sigma}(t)}Q\left(\rho\right)e^{i\tau+i\alpha\frac{L_t}{4L}r^2+i(1-\alpha)\frac{L_t}{4L}(r-r_{max}(
t))^2},
\end{equation}
\begin{equation}
\tau=\int_0^t\frac{ds}{L^2(s)},\qquad\rho=\frac{r-r_{max}( t)}{L(t)},\qquad
r_{max}( t)=r_0 L^\alpha( t),
\end{equation}
and
\begin{equation}\label{eq:psiQ_alpha}
\alpha = \frac{2-\sigma}{\sigma(d-1)}.
\end{equation}
\end{subequations}
The self-similar profile~$Q$ attains its global maximum at~$\rho=0$.
Hence,~$r_{max}(t)$ is the ring radius and~$L(t)$ is the ring width,
see Figure~\ref{fig:rm_def}.
\begin{figure}
\begin{center}
      \centerline{\hbox{\vspace{-5cm}
                  \scalebox{.6}{\includegraphics{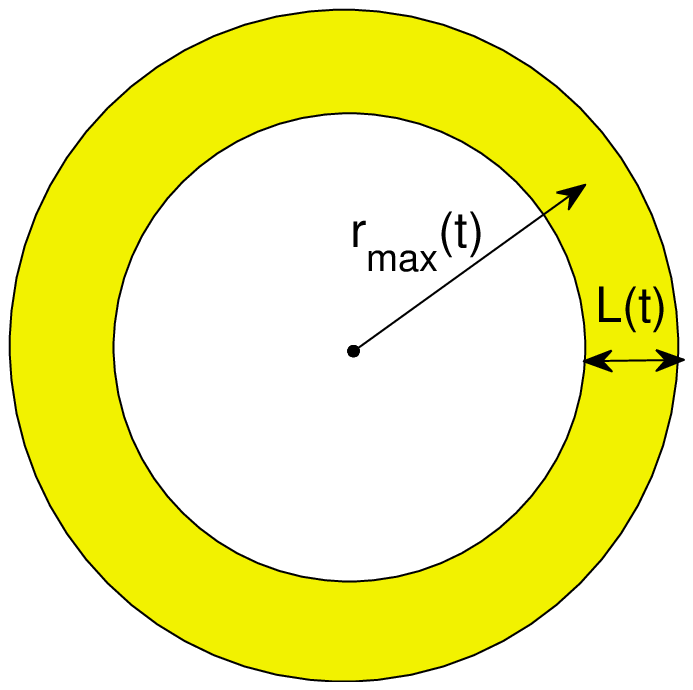}}
                  \hspace{0.5cm}
                  \scalebox{.65}{\includegraphics{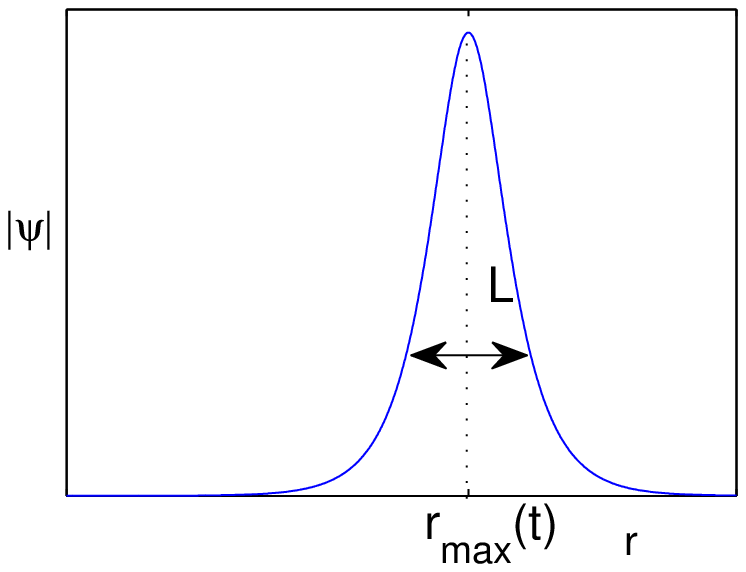}
                  }}}
 \mycaption{Illustration of ring radius~$\rmax( t)$ and width~$L( t)$.\label{fig:rm_def}}
\end{center}
\end{figure}

The~$\psi_Q$ ring solutions can be classified as follows, see
Figure~\ref{fig:NLS_class_diagram}:
\begin{enumerate}
    \renewcommand{\theenumi}{\Alph{enumi}}
    \item In the  subcritical case~($\sigma d<2$), all NLS solutions globally
            exist, hence no singular ring solutions exist.
    \item The critical case~$\sigma d=2$ corresponds to~$\alpha=1$.
            Since~$\rmax( t)=r_0 L( t)$, these solutions undergo
            an {\em equal-rate collapse}, i.e., the ring radius
            goes to zero at the same rate as~$L( t)$. The blowup
            rate of these critical ring solutions is a square
            root.
    \item The supercritical case~$2/d<\sigma<2$ corresponds to $0<\alpha<1$.
            Therefore, the ring radius~$\rmax( t)=r_0 L^\alpha(
            t)$ decays to zero, but at a slower rate than~$L(
            t)$. The blowup rate of these ring solutions is
                \begin{equation}\label{eq:blowup_rate_with_p}
                    L(t)\sim \kappa(T_c-t)^p,
                \end{equation}
where~$p=\frac{1}{1+\alpha}=\frac{\sigma(d-1)}{2+\sigma(d-2)}$.
    \item The supercritical case $\sigma=2$ corresponds  to $\alpha=0$.
            Since~$\rmax(t)\equiv r_0$, the solution becomes
            singular on the d-dimensional sphere~$|{\bf
            x}|=r_0$, rather than at a point.  The blowup
            rate of these solutions is given by the loglog law~\eqref{eq:logloglaw}.
    \item The case~$\sigma>2$ was open until now.
\end{enumerate}
\begin{figure}[h]
    \begin{center}
           \scalebox{0.7}{
        \input{NLS_ring_class.pstex_t}}
        \mycaption{\label{fig:NLS_class_diagram}
        Classification of singular ring solutions of the NLS as a function of
        $\sigma$ and $d$.
        A:~subcritical case.
        B:~critical case, with equal-rate collapse~\cite{Gprofile-05}.
        C:~$2/d<\sigma<2$, shrinking rings~\cite{SC_rings-07}.
        D:~$\sigma=2$, standing rings~\cite{SC_rings-07,Raphael-06,Raphael-08}.
        E:~the case $\sigma>2$, which is considered in this study.
        }
    \end{center}
\end{figure}
Thus, the~$\psi_Q$ solutions are {\em shrinking rings}
(i.e.,~$\lim_{t\to T_c}{r_{max}(t)}=0$) for~$\frac2d\le \sigma <2$
(cases B and C), and {\em standing rings} (i.e.,~$0<\lim_{t\to
T_c}{r_{max}(t)}<\infty$) for~$\sigma=2$ (case D).

\subsection{Singular standing-ring solutions of the NLS}
One of the goals of this paper is to study singular ring-type
solutions for $\sigma>2$ (case E). The most natural guess is that
these solutions also blowup with the~$\psi_Q$ profile.
Since~$\alpha<0$ for~$\sigma>2$, see~\eqref{eq:psiQ_alpha},
$\psi_Q$~should be an {\em expanding ring}, i.e.,~$\lim_{t\to
T_c}{r_{max}}=\infty$. In this study we show that although such
expanding rings do not violate power conservation,~$\psi_Q$ ring
solutions cannot exist for~$\sigma>2$. Rather,  singular rings
solutions of the NLS~\eqref{eq:radial-NLS} with~$\sigma>2$ are
standing rings.

The blowup profile and rate of standing-ring solutions can be
obtained using the following informal argument. In the ring region
of a standing-ring,~$\psi_{rr}\sim\frac1{L^2}$
and~$\frac1r{\psi_r}\sim \frac1{L\cdot r_{max}}$. Therefore,
the~$\frac{d-1}r\psi_{r}$ term in equation~\eqref{eq:radial-NLS}
becomes negligible compared with~$\psi_{rr}$ as~$t\to T_c$. Hence,
near the singularity, equation~\eqref{eq:radial-NLS} reduces to the
one-dimensional NLS\footnote{Throughout this paper, we denote the
solution of the one-dimensional NLS by~$\phi$, and its spatial
variable by~$x$.}
\[
i\phi(t,x)+\phi_{xx}+|\phi|^{2\sigma}\phi=0,\quad x=r-r_{max}.
\]
Therefore, the blowup profile and blowup-rate of standing-ring
solutions of the NLS~\eqref{eq:radial-NLS} with $d>1$
and~$\sigma\ge2$ are the same as those of singular peak-type
solutions of the one-dimensional NLS with the same~$\sigma$.
Specifically, standing-ring solutions are self-similar in the ring
region, i.e.,~$\psi\sim\psi_F$, where~$\psi_F$ is, up to a shift
in~$r$, the asymptotic peak-type profile of the one-dimensional NLS,
i.e.,
\[
        \psi_F(t,r-r_{max})=\phi_S( t,x ),
\]
and~$\phi_S$ is given by~\eqref{eq:intro_psiS_profile} with~$d=1$.
In addition, the blowup rate~$L(t)$ of~$\psi_F$ is the same as the
blowup rate of~$\phi_S$, see~\eqref{eq:intro_psiS_blowuprate}, and
is equal to a square-root, i.e.,
\begin{equation}\label{eq:sqrt_blowup_rate}
L(t)\sim \kappa(\sigma)\sqrt{T_c-t},\qquad t\to T_c.
\end{equation}
Moreover,~$\kappa(\sigma)$ is universal, i.e., it depends
on~$\sigma$, but not on the dimension~$d$ or the initial
condition~$\psi_0$.

Numerically, we observe that~$\psi_F$ is an attractor for a large
class of radially-symmetric initial conditions, but is unstable with
respect to symmetry-breaking perturbations.
\subsection{Singular standing vortex solutions of the NLS}
The two-dimensional NLS
\begin{equation*}
i\psi_t(t,x,y)+\Delta \psi+|\psi|^{2\sigma}\psi=0,\qquad \psi(0,x,y)=\psi_0(x,y),\qquad \Delta=\partial_{xx}+\partial_{yy},
\end{equation*}
admits vortex solutions of the form $
   \psi(t,x,y) = A(t,r)e^{im\theta}$, where~$m\in\mathbb{Z}$.
In~\cite{Vortex_rings-08}, we presented a systematic study of
singular vortex solutions.  In particular, we showed that there
exist singular vortex solutions that collapse with the asymptotic
profile~$\psi_Q\cdot e^{im\theta}$, when~$\psi_Q$ is given
by~\eqref{eq:psiQ}. The blowup rates of these solutions are the same
as those of~$\psi_Q$ in the non-vortex case. Therefore,
the~$\psi_Q\cdot e^{im\theta}$ vortex solutions can be classified as
follows:
\begin{enumerate}
    \renewcommand{\theenumi}{\Alph{enumi}}
    \item In the  subcritical case~($\sigma<1$), all NLS solutions globally
            exist, hence no singular vortex solutions exist.
    \item In the critical case~$\sigma=1$, these solutions undergo
            an {\em equal-rate collapse} at a square root blowup
            rate.
    \item The supercritical case~$1<\sigma<2$ corresponds to $0<\alpha<1$.
            In this case, the ring radius decays to zero at a
            slower rate than~$L( t)$ and the blowup rate is
            given by~\eqref{eq:blowup_rate_with_p}
            where~$p=\frac{1}{1+\alpha}=\frac\sigma2$.
    \item The supercritical case $\sigma=2$ corresponds to~$\alpha=0$.  Therefore, the
            solution becomes singular on a circle. The blowup
            rate is given by the loglog
            law~\eqref{eq:logloglaw}.
    \item The case~$\sigma>2$ was open until now.
\end{enumerate}

In this study, we show numerically that there exist singular standing-vortex
solutions of the two-dimensional NLS with~$\sigma>2$ (case E) with
the asymptotic profile~$e^{im\theta}\psi_F$. Moreover, the blowup
rate of these singular standing vortices is the same as that of the
standing-ring solutions in the non-vortex case, i.e., is given
by~\eqref{eq:sqrt_blowup_rate}. Therefore, these results extend the
ones obtained in the non-vortex case.

\subsection{Singular solutions of the biharmonic nonlinear Schr\"odinger equation}
Let us consider the focusing biharmonic nonlinear Schr\"odinger equation (BNLS)
equation
\begin{equation}    \label{eq:BNLS}
    i\psi_t(t,r) - \Delta^2\psi + \left|\psi\right|^{2\sigma}\psi = 0,
\end{equation}
where~$\Delta^2$ is the radial biharmonic operator. Here, singularly
formation is defined as~${\displaystyle \lim_{t\to
T_c}}\norm{\psi}_{H^2}=\infty$. In the {\em subcritical} case
$\sigma d<4$, all BNLS solutions exists
globally~\cite{Fibich_Ilan_George_BNLS:2002}. Numerical
simulations~\cite{Fibich_Ilan_George_BNLS:2002,Baruch_Fibich_Mandelbaum:2009}
indicate that in the {\em critical} case $\sigma=4/d$ and the {\em
supercritical} case $\sigma\geq4/d$, the BNLS admits singular
solutions. At present, however, there is no rigorous proof that the
BNLS admits singular solutions, whether peak-type or ring-type.

Peak-type singular solutions of the BNLS~\eqref{eq:BNLS} were
recently studied numerically
in~\cite{Baruch_Fibich_Mandelbaum:2009}. The blowup rate of these
solutions is slightly faster than~$p=1/4$ in the critical case
($1/4+$loglog?), and is equal to~$p=1/4$ in the supercritical case.

The BNLS~\eqref{eq:BNLS} also admits ring-type singular solutions
for~$4/d\leq \sigma \leq 4$~\cite{Baruch_Fibich_Mandelbaum:2009}.
These solutions are of the form~$\psi\sim\psi_{Q_{\tiny B}}$,
where
\begin{subequations}
\label{eq:psiQ_BNLS}
\begin{equation}
\label{}
|\psi_{Q_{\tiny{B}}}|=\frac{1}{L^{1/{2\sigma}}(t)}Q_{\tiny{
B}}\left(\rho\right),
\end{equation}
\begin{equation}
\qquad\rho=\frac{r-r_{max}( t)}{L(t)},\qquad r_{max}( t)=r_0
L^{\alpha_{\tiny{B}}}( t),
\end{equation}
and
\begin{equation}\label{eq:psiQ_alpha_BNLS}
\alpha_{\tiny{B}} = \frac{4-\sigma}{\sigma(d-1)}.
\end{equation}
\end{subequations}
The~$\psi_{Q_B}$ solutions can be classified as follows (see
Figure~\ref{fig:BNLS_class_diagram}):
\begin{enumerate}
    \renewcommand{\theenumi}{\Alph{enumi}}
    \item In the subcritical case~($\sigma d<4$), all BNLS solutions globally
            exist, hence no collapsing ring solutions exist.
    \item The critical case~$\sigma d=4$ corresponds
    to~$\alpha_{\tiny{B}}=1$. These solutions undergo an {\em
            equal-rate collapse}. The blowup rate of these
            critical ring solutions is given
            by~\eqref{eq:blowup_rate_with_p} with~$p=1/4$.
    \item The supercritical case~$4/d<\sigma<4$ corresponds to $0<\alpha_{\tiny
B}<1$. Therefore, the ring radius~$\rmax( t)=r_0 L^\alpha_{\tiny
            B}(t)$ decays to zero, but at a slower rate than~$L(
            t)$. The blowup rate of these ring solutions is
            given by~\eqref{eq:blowup_rate_with_p}
            with~$p=1/(3+\alpha_{\tiny
B})=\sigma(d-1)/(4+3\sigma
           d-4\sigma)$.
    \item The case $\sigma=4$ corresponds to $\alpha_{\tiny
B}=0$. Since~$\rmax(t)\equiv r_0$, the solution is a singular
            standing ring. The blowup rate is close to~$p=1/4$
            and is conjectured to be~$1/4$ with a loglog
            correction.
    \item The case~$\sigma>4$ was open until now.
\end{enumerate}
\begin{figure}[h]
    \begin{center}
            \scalebox{0.7}{
            \input{BNLS_ring_class.pstex_t}}
        \mycaption{\label{fig:BNLS_class_diagram}
        Classification of singular ring solutions of the BNLS as a function of
        $\sigma$ and $d$.
        A: subcritical case.
        B: critical case, with equal-rate collapse~\cite{Baruch_Fibich_Mandelbaum:2009}.
        C: $4/d<\sigma<4$, shrinking rings~\cite{Baruch_Fibich_Mandelbaum:2009}.
        D: $\sigma=4$, standing rings with the critical $1D$
        profile~\cite{Baruch_Fibich_Mandelbaum:2009}.
        E: the case $\sigma>4$, which is considered in this study.
        }
    \end{center}
\end{figure}
Thus, up to the change~$\sigma\longrightarrow 2\sigma$, this
classification is completely analogous to that of singular ring
solutions of the NLS (see Figure~\ref{fig:NLS_class_diagram}).
In this work we show numerically that this analogy carries through to the
regime~$\sigma>4$.  Thus, the BNLS with~$\sigma>4$ and~$d>1$ admits
singular standing-ring solutions. Near the standing-ring peak,
equation~\eqref{eq:BNLS} reduces to the one-dimensional BNLS
\begin{equation*}
    i\phi_t(t,x)-\phi_{xxxx}+|\phi|^{2\sigma}\phi=0.
\end{equation*}
Therefore, the blowup profile and blowup-rate of standing ring
solutions of the BNLS~\eqref{eq:BNLS} with $d>1$ and~$\sigma\ge4$
are the same as those of collapsing peak solutions of the
one-dimensional BNLS with the same value of~$\sigma$.

Thus, the results for ring solutions of the BNLS with~$\sigma>4$ are
completely analogous, up to the change~$\sigma\longrightarrow
2\sigma$, to those for singular standing-ring solutions of the NLS
with~$\sigma>2$.

\subsection{Singular solutions of the nonlinear heat equation and the biharmonic nonlinear heat equation}
The $d$-dimensional nonlinear heat equation (NLHE)
\begin{equation}    \label{eq:intro_NLHE}
u_t(t,r)-\Delta u-|u|^{2\sigma}u=0,\qquad \sigma>0,\quad d>1,
\end{equation}
and the $d$-dimensional nonlinear {\em biharmonic} heat equation (NLBHE)
\begin{equation}    \label{eq:intro_NLBHE}
u_t(t,r)+\Delta^2 u-|u|^{2\sigma}u=0,\qquad \sigma>0,\quad d>1.
\end{equation}
admit singular solutions for
any~$\sigma>0$~\cite{GigaKohn_85,Budd-Williams-Galaktionov_2004}. To
the best of our knowledge, until now all known singular solutions
of~\eqref{eq:intro_NLHE} and of~\eqref{eq:intro_NLBHE} collapse at a
point (or at a finite number of points~\cite{Merle_Heat_1992}).
In this study, we provide numerical evidence that the NLHE~\eqref{eq:intro_NLHE}
and the NLBHE~\eqref{eq:intro_NLBHE} admit singular standing-ring solutions.
The blowup profile and blowup rate of these solutions are the same as those of
singular peak-type solutions of the corresponding one-dimensional equation.

\subsection{Critical exponents of singular ring solutions}\label{sec:discussion}
In Figure~\ref{fig:blowup_rate_summary}A we plot the blowup rate
parameter~$p$ of singular ring solutions of the NLS,
see~\eqref{eq:blowup_rate_with_p}. As~$\sigma$ increases from~$2/d$
to~$2-$,~$p$ increases monotonically from~$1/2$ to~$1-$.
When~$\sigma=2$, the blowup rate is given by the loglog
law~\eqref{eq:logloglaw}, i.e.,~$p=1/2$ with a loglog correction.
Finally,~$p=1/2$~for~$\sigma>2$.
\begin{figure}
    \begin{center}
    \scalebox{.8}{\includegraphics{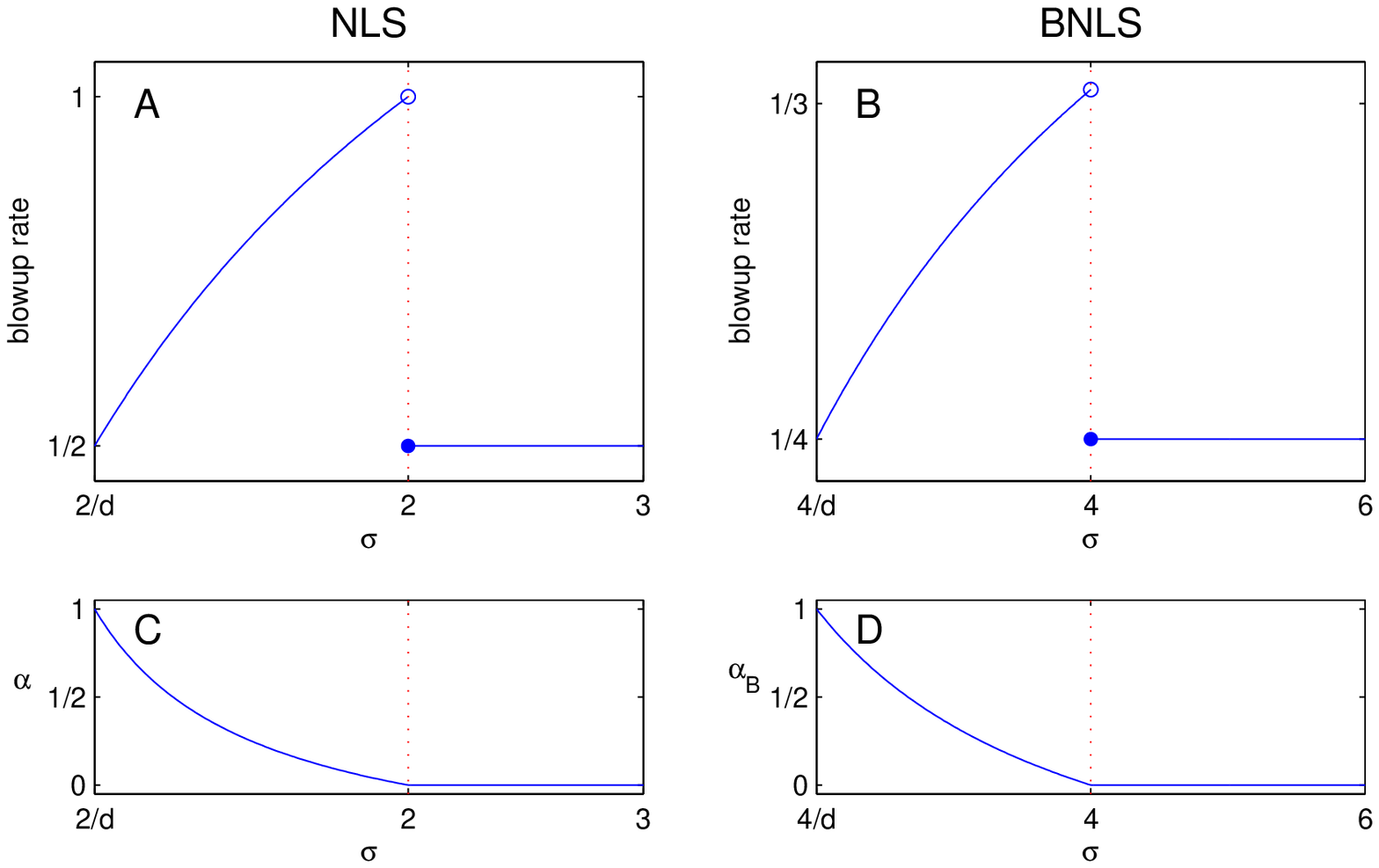}}
    \mycaption{
A: Blowup rate of singular ring solutions of the NLS.
The blowup rate increases monotonically from~$p=1/2$ at~$\sigma=2/d$ to~$p=1-$
at~$\sigma=2-$. For~$\sigma = 2$
(full circle)~$p=1/2$ (with a loglog correction) and for~$\sigma>2$,~$p\equiv1/2$.
B: Blowup rate of singular ring solutions of the BNLS.
The blowup rate increases monotonically from~$p=1/4$ at~$\sigma=4/d$ to~$p=(1/3)-$ at~$\sigma=4-$.
For~$\sigma = 4$ (full circle)~$p=1/4$ (with a loglog correction?) and for~$\sigma>4$,~$p\equiv1/4$.
C: The shrinkage parameter~$\alpha$, defined by the relation~$r_{max}\sim r_0L^\alpha$ of singular ring solutions of the NLS.  For~$2/d\le\sigma<2$,~$\alpha$ decreases monotonically from 1 to 0+ (shrinking rings). For~$\sigma\ge2$,~$\alpha\equiv0$ (standing rings).
D: The shrinkage parameter~$\alpha_B$ of singular ring solutions of the BNLS.  For~$4/d\le\sigma<4$,~$\alpha_B$ decreases monotonically from 1 to 0+ (shrinking rings). For~$\sigma\ge4$,~$\alpha_B\equiv0$ (standing rings).
     \label{fig:blowup_rate_summary}
    }
    \end{center}
\end{figure}
Since
\[
\lim_{\sigma\to2-}p=1,\qquad\lim_{\sigma\to2+}p=\frac12,
\]
the blowup rate has a discontinuity at~$\sigma=2$. Surprisingly, the
blowup rate is not monotonically-increasing with~$\sigma$. For
example, a ring solution of the NLS with~$\sigma=1.8$ blows up
faster than a ring solution of the NLS with~$\sigma=2.2$.

The above results show that
{\em the critical exponent of singular ring solutions of the NLS is $\sigma=2$}:
The blowup rate is discontinuous at~$\sigma=2$, and the blowup dynamics changes from a
shrinking-ring~$(\sigma<2)$ to a standing-ring~$(\sigma\ge2)$, see
Figure~\ref{fig:blowup_rate_summary}C. We can understand
why~$\sigma=2$ is a critical exponent using the following argument.
Standing-ring solutions are `equivalent' to singular peak solutions
of the one-dimensional NLS with the same nonlinearity
exponent~$\sigma$. Since~$\sigma=2$ is the critical exponent for
singularity formation in the one-dimensional NLS, it is also a
critical exponent for standing-ring blowup.

An analogous picture exists for the BNLS. For~$4/d\le\sigma<4$, the
blowup-rate~$p$ of the BNLS ring solutions increases monotonically
in~$\sigma$ from~$1/4$ to~$(1/3)-$, at~$\sigma=4$,~$p=1/4$ possibly
with a loglog correction, and~$p=1/4$ for~$\sigma>4$, see
Figure~\ref{fig:blowup_rate_summary}B.
Therefore, the blowup rate is discontinuous at~$\sigma=4$.
In addition, the blowup dynamics change at~$\sigma=4$ from a
shrinking-ring~$(\sigma<4)$ to a standing-ring~$(\sigma>4)$, see
Figure~\ref{fig:blowup_rate_summary}D.
Hence, {\em the critical exponent of standing-ring solutions of the BNLS is
$\sigma=4$},
precisely because it is the critical exponent for singularity
formation in the one-dimensional BNLS.

In the case of the NLHE and BNLHE equations, there is no critical
exponent of singular ring solutions.
Indeed, these equations admit standing-ring solution for any~$\sigma>0$,
precisely because there is no critical exponent for singularity formation
in the corresponding one-dimensional equations.
\subsection{Paper outline}\label{sec:outline}
The paper is organized as follows. In
Section~\ref{sec:1D_NLS_collapse} we review the theory of singular
peak-type solutions of the supercritical NLS, and conduct a
numerical study of the one-dimensional case. In
Section~\ref{sec:theory_for_NLS_standing} we prove that
standing-ring blowup can only occur for~$\sigma\geq2$, and show that
the blowup profile and blowup-rate of singular standing ring
solutions of the supercritical NLS with~$\sigma>2$ and~$d>1$ are the
same as those of peak-type solutions of the one-dimensional NLS
equation. In Section~\ref{sec:simulations_for_NLS_standing} we
confirm these results numerically. We then show numerically that the
singular standing-ring profile~$\psi_F$ is an attractor for
radially-symmetric initial
conditions~(Section~\ref{sec:robustness}), but it is unstable with
respect to symmetry-breaking
perturbations~(Section~\ref{sec:instablilty_symmetry_breaking}). In
Section~\ref{sec:NLS_non-standing_rings} we show analytically and
numerically that expanding~$\psi_Q$ ring solutions do not exist
for~$\sigma>2$. Section~\ref{sec:vortex} extends the results to
singular vortex solutions.  In Section~\ref{sec:1D_BNLS_collapse} we
study singular peak-type solutions of the one-dimensional
supercritical BNLS. In Section~\ref{sec:BNLS_standing} we show that
singular ring solutions of the supercritical BNLS with~$\sigma>4$
are standing-rings, whose blowup profile and blowup-rate are the
same as those of peak-type solutions of the one-dimensional BNLS. In
Section~\ref{sec:NLH} we show that singular standing-ring solutions
of the nonlinear heat equation exist for any~$\sigma>0$, and that
their blowup profile and blowup-rate are the same as those of
peak-type solutions of the one-dimensional NLHE. In
Section~\ref{sec:BNLHE} we show that singular standing-ring
solutions of the nonlinear biharmonic heat equation exist for
any~$\sigma>0$, and that their blowup profile and blowup-rate are
the same as those of peak-type solutions of the one-dimensional
BNLHE. The numerical methods used in this study are briefly
described in Section~\ref{sec:numerical_methods}.
\section{Singular peak-type solutions of the one-dimensional supercritical NLS}\label{sec:1D_NLS_collapse}
\subsection{Theory review}\label{sec:theory_for_1D_NLS_collapse}
Let us consider the one-dimensional supercritical NLS
\begin{equation}\label{eq:1D_SC_NLS}
    i\phi_t(t,x)+\phi_{xx}+|\phi|^{2\sigma}\phi=0,  \qquad \sigma >2.
\end{equation}
In contrast to the extensive theory on singularity formation in the
critical NLS, much less is known about the supercritical case.
Previous numerical simulations and formal calculations (see,
e.g.,~\cite[Chapter 7]{Sulem-99} and the references therein)
suggested that peak-type singular solutions of the supercritical
NLS~(\ref{eq:1D_SC_NLS}) collapse with a self-similar asymptotic
profile~$\phi_S$, .i.e.,~$\phi\sim\phi_S$, where
\begin{equation}
\phi_S(
t,x)=\frac{1}{L^{1/\sigma}(t)}S\left(\xi\right)e^{i\tau+i\frac{L_t}{4L}x^2},
\qquad\xi=\frac{x}{L(t)},\qquad \tau=\int_0^t\frac{ds}{L^2(s)}. \label{eq:psiS}
\end{equation}
The blowup rate~$L(t)$ of these solutions is a square root,
i.e.,
\begin{equation}
L( t)\sim \LC\sqrt{ T_c- t},\qquad  t\to T_c,
\label{eq:intro_square_root_blowuprate}
\end{equation} where $\LC>0$.
In addition, the self-similar profile~$S$ is the solution of
\begin{eqnarray}\label{eq:ODE4S}
    S^{\prime\prime}(\xi) -
    \left(1 + i\frac{\sigma-2}{4\sigma}\LC^2 -\frac{\LC^4}{16}\xi^2     \right) S
    + |S|^{2\sigma}S = 0,     \qquad S'(0)=0,\qquad S(\infty)=0.
\end{eqnarray}
In general, solutions of~(\ref{eq:ODE4S}) are complex-valued, and
depend on the parameter~$\LC$ and on the initial condition~$S(0)$.
Solutions of~(\ref{eq:ODE4S}) whose amplitude~$|S|$ is
monotonically-decreasing in~$\xi$, and which have a zero
Hamiltonian, are called {\em admissible solutions}~\cite{Sulem-99}.
For each~$\sigma$, equation~\eqref{eq:ODE4S} has a unique admissible
solution (up to a multiplication by a constant phase~$e^{i\alpha}$).
This solution is attained for specific real values of~$\LC$ and~$S(0)$,
which we denote as
\begin{equation}\label{eq:unique_parms_ODE4S}
\LC=\LC_S(\sigma),\qquad S(0)=S_0(\sigma).
\end{equation}
Moreover, numerical simulations and formal calculations suggest that:
\begin{enumerate}
    \item The self-similar profile of singular peak-type solutions of the
            NLS~(\ref{eq:1D_SC_NLS}) is an admissible solution
            of~\eqref{eq:ODE4S}.
    \item The constant $\LC$ of the blowup
            rate~\eqref{eq:intro_square_root_blowuprate} is
            universal (i.e., is independent of the initial
            condition~$\psi_0$), and is equal to
            $\LC_S(\sigma)$.
\end{enumerate}

\subsection{Simulations}
To the best of our knowledge, the theory of supercritical peak-type collapse
which is presented in Section~\ref{sec:theory_for_1D_NLS_collapse},
was tested numerically only for~$d\ge2$.  Since this theory is not
rigorous, and since we will make use of these results in
Sections~\ref{sec:NLS_standing}
and~\ref{sec:NLS_non-standing_rings}, we now confirm numerically the
above theoretical predictions for the one-dimensional supercritical
NLS~\eqref{eq:1D_SC_NLS}.

\begin{figure}
    \begin{center}
    \scalebox{.8}{\includegraphics{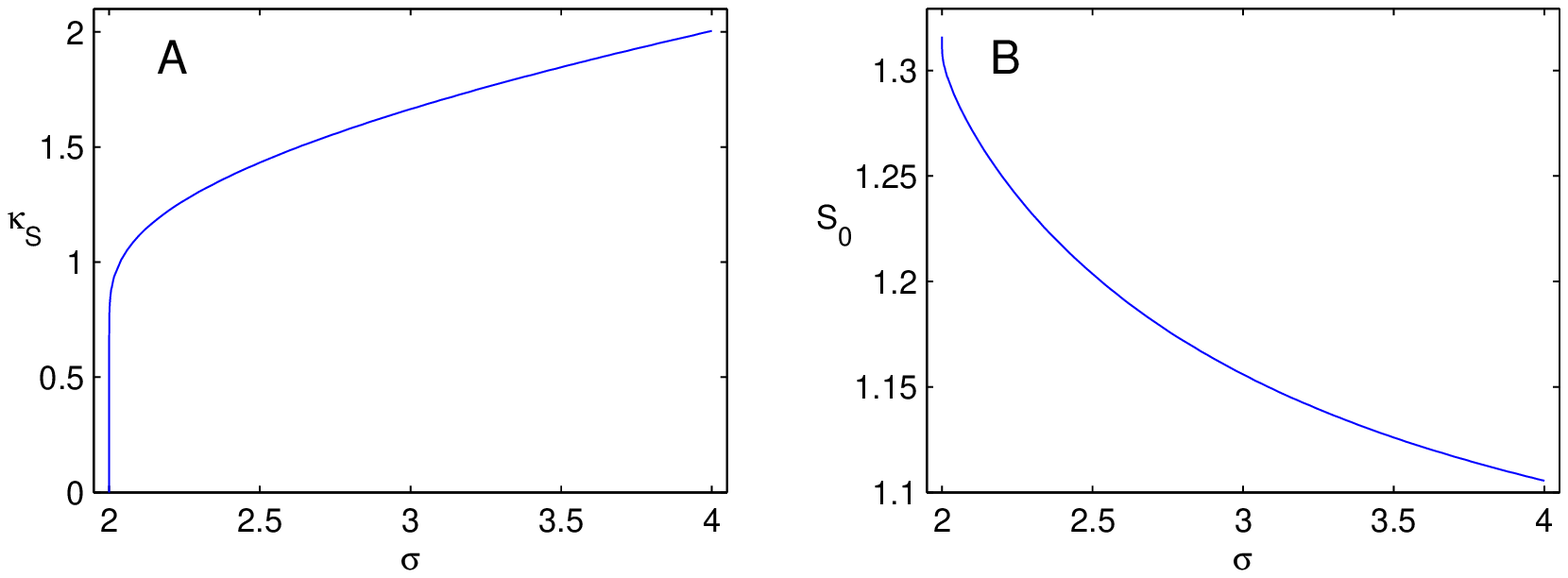}}
    \mycaption{
        The parameters~$\kappa_S$ and~$S_0$ of the
        admissible solutions of equation~(\ref{eq:ODE4S}), as function of~$\sigma$.
        \label{fig:fc_as_fuction_of_sigma}
    }
    \end{center}
\end{figure}

We verified numerically that for each~$\sigma$, there exists a
unique admissible solution of~\eqref{eq:ODE4S}. The corresponding
values of~$\kappa_S$ and~$S_0$ as a function of~$\sigma$ are shown
in Figure~\ref{fig:fc_as_fuction_of_sigma}.
For example, for~$\sigma=3$, the parameters of the admissible
solution of equation~(\ref{eq:ODE4S}) are
\begin{equation}\label{eq:ODE4s_parms_sigma=3}
    \LC_S(\sigma=3)\approx1.664,\qquad  S_0(\sigma=3)\approx1.155.
\end{equation}

We now solve the one-dimensional NLS~(\ref{eq:1D_SC_NLS})
with~$\sigma=3$ and the Gaussian initial
condition~$\left.\phi_0(t=0,x)=2e^{-2x^2}\right.$. We first show
that the NLS solution collapses with the self-similar
profile~$\phi_S$, see~\eqref{eq:psiS}. To do that, we rescale the
solution according to
\begin{equation}\label{eq:normalization}
\phi_{\rm
rescaled}(t,x)=L^{\frac1\sigma}(t)\left|\phi\left(\frac{x}{L(t)}\right)\right|\,\qquad
L(t)=\left(\frac{\|S\|_\infty}{\|\phi\|_\infty}\right)^\sigma=\left(\frac{S_0^{1D}(\sigma)}{\|\phi\|_\infty}\right)^\sigma.
\end{equation}
The rescaled profiles at focusing levels of~$1/L=10^4$ and~$1/L =
10^{8}$ are indistinguishable, see
Figure~\ref{fig:NLS_d=1_sigma=3}A, indicating that the solution is
indeed self-similar while focusing over 4 orders of magnitude.
Moreover, the rescaled profiles are in perfect fit with the
admissible~$S(\xi;\sigma=3)$ profile.
\begin{figure}
    \begin{center}
    \scalebox{.8}{\includegraphics{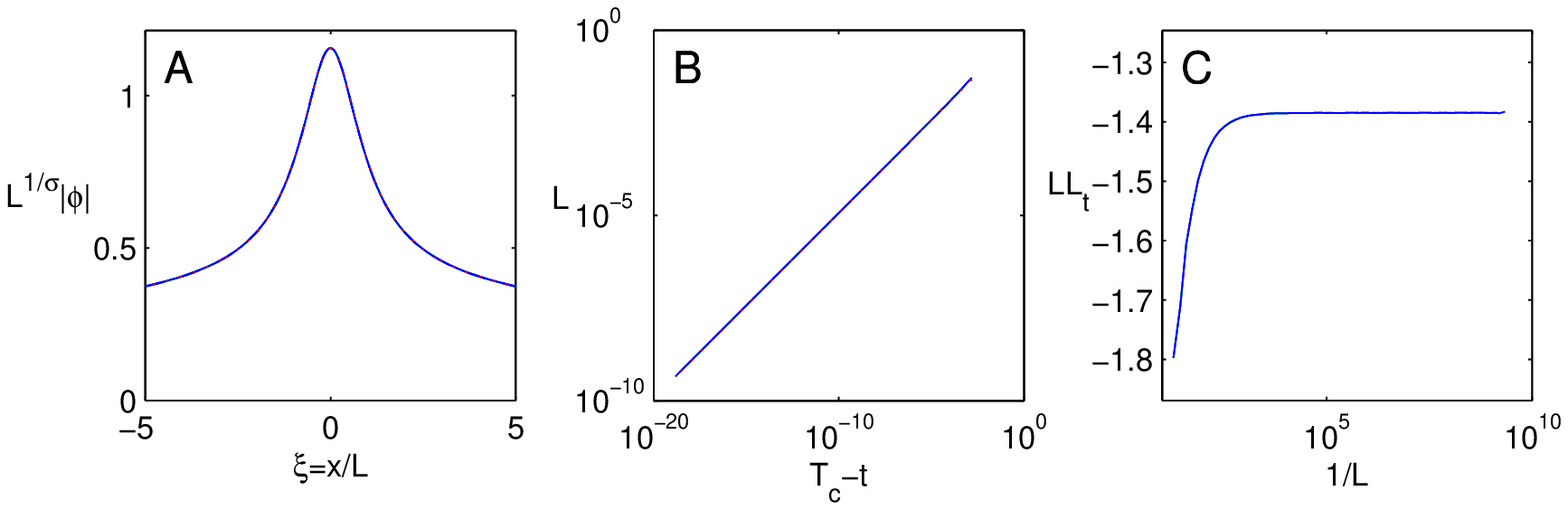}}
    \mycaption{
        Solution of the one-dimensional NLS~(\ref{eq:1D_SC_NLS}) with $\sigma=3$ and
        the initial condition~$\psi_0=2e^{-2x^2}$.
        A:~Rescaled solution, see~(\ref{eq:normalization}), at the focusing
        levels $1/L=10^4$ (solid) and~$1/L=10^8$ (dashed); dotted curve is the admissible
        profile~$S(\xi;\sigma=3)$, all three curves are indistinguishable.
        B:~$L$ as a function of~$(T_c-t)$ on a logarithmic scale.
        Dotted curve is the fitted curve~$1.713\cdot( T_c-t)^{0.5007}$.
        The two curves are indistinguishable.
        C:~$LL_t$ as a function of $1/L$.    \label{fig:NLS_d=1_sigma=3}
    }
    \end{center}
\end{figure}

Next, we consider the blowup rate of $\phi$. To do
that, we first assume that
\begin{equation}\label{eq:L_assumption}
L(t)\sim \kappa(T_c-t)^p,
\end{equation}
and find the best fitting~$\kappa$ and~p, see
Figure~\ref{fig:NLS_d=1_sigma=3}B. In this case~$\kappa\approx
1.713$ and~$p \approx 0.5007$, indicating that the blowup rate
is~square-root or slightly faster.
In order to check whether~$L$ is slightly faster than a square root, we
compute the limit~$\displaystyle{\lim_{t\to T_c}}LL_t$, see~\cite{Gprofile-05}.
Recall that for the square-root blowup
rate~\eqref{eq:intro_square_root_blowuprate},
\[
\lim_{t\to T_c}LL_t=-\frac{\kappa^2}2<0,
\]
while for a faster-than-a-square root blowup rate~$LL_t$ goes to
zero. Since $ {\displaystyle\lim_{T_c\to t}}LL_t=-1.384$,
see Figure~\ref{fig:NLS_d=1_sigma=3}C, the blowup rate of~$\phi$ is
square-root (with no loglog correction), i.e.,
\begin{equation}\label{eq:kappa_blowup_1D}
L(t)\sim \LC^{\rm blowup-rate}_{1D}\sqrt{T_c-t},\qquad \LC_{1D}^{\rm
blowup-rate}\approx\sqrt{2\cdot1.384} \approx 1.664.
\end{equation}
In particular, there is an excellent match (to 4 digits)
between~$\LC^{\rm blowup-rate}_{1D}\approx 1.664$ extracted from the blowup rate
of~$\phi$, see~\eqref{eq:kappa_blowup_1D} and the parameter~$\LC_S(\sigma=3)$ of the
admissible~$S(\xi;\sigma=3)$ profile,
see~\eqref{eq:ODE4s_parms_sigma=3}.

\section{Singular standing-ring solutions of the supercritical NLS}
\label{sec:NLS_standing}

Let us consider singular solutions of the supercritical NLS
\begin{equation}\label{eq:NLS}
    i\psi_t(t,r)+\psi_{rr}+\frac{d-1}{r}\psi_r+|\psi|^{2\sigma}\psi=0,
    \qquad d>1,\qquad \sigma d>2.
\end{equation}
In this Section we show that equation~\eqref{eq:NLS} has singular
standing-ring solutions for $\sigma\ge2$. Since the case $\sigma=2$
was already studied in~\cite{SC_rings-07,Raphael-06,Raphael-08}, we mainly focus
on the case $\sigma>2$.

\subsection{Analysis}\label{sec:theory_for_NLS_standing}

The following Lemma shows that standing-ring blowup can only occur
for~$\sigma\geq2$:
\begin{lem} \label{lem:NLS_sigma_ge_2}
    Let~$\psi$ be a standing-ring singular solution of the NLS~\eqref{eq:NLS},
    i.e.,~$\psi\sim\psi_F$ for $\abs{r-\rmax}\le\rho_c\cdot L(t)$, where
    \begin{subequations}
    \begin{equation}\label{eq:abs_psiF}
    |\psi_F(t,r)|=\frac{1}{L^{1/\sigma}(t)}|F\left(\rho\right)|,\qquad \rho=\frac{r-r_{max}(t)}{L(t)},\qquad \lim_{t\to T_c}L(t)=0,
    \end{equation}
    and
    \begin{equation}
        0< \lim_{t\to T_c}\rmax(t) < \infty.
   \end{equation}
    \end{subequations}
    Then,~$\sigma\ge2$.
    \begin{proof} The power of the collapsing part~$\psi_F$ is
    \begin{eqnarray*}
        \norm{\psi_F}_2^2 &=&
            L^{-2/\sigma}\int_{r=\rmax-\rho_c\cdot L(t)}^{\rmax+\rho_c\cdot L(t)}
            \abs{F\left( \frac{r-\rmax}{L} \right)}^2r^{d-1}dr \\
        &=&
            L^{-2/\sigma}\int_{\rho=-\rho_c}^{\rho_c}
            \abs{F(\rho)}^2(L\rho+\rmax)^{d-1}(Ld\rho) \\
                &\sim& L^{1-2/\sigma}(t) \cdot \rmax^{d-1}
                    \int_{\rho=-\rho_c}^{\rho_c} \abs{F(\rho)}^2d\rho.
    \end{eqnarray*}
    Since $
        \norm{\psi_F}_2^2
        \le \norm{\psi}_2^2
        = \norm{\psi_0}_2^2<\infty,
    $
    then $L^{1-2/\sigma}$ has to be bounded as $L\to0$, hence~$\sigma\ge2$.
    \end{proof}
\end{lem}

Let
\[
P_{\rm collapse}=\liminf_{\varepsilon\to 0+}\lim_{t\to
T_c}\int_{|r-r_{max}(t)|<\varepsilon}|\psi|^2r^{d-1}dr
\]
be the amount of power that collapses into the standing-ring
singularity. We say that a singular standing-ring solution~$\psi$
undergoes a {\em strong collapse} if~$P_{\rm collapse}>0$, and a {\em weak
collapse} if~$P_{\rm collapse}=0$.

\begin{cor} \label{cor:standing_NLS_strength}
Under the conditions of Lemma~\ref{lem:NLS_sigma_ge_2},~$\psi_F$
undergoes a strong collapse when~$\sigma=2$, and a weak collapse
when~$\sigma>2$.
\end{cor}
\begin{proof} This follows directly from the proof of
Lemma~\ref{lem:NLS_sigma_ge_2}.\end{proof}

Let us further consider singular standing-ring solutions of the
NLS~(\ref{eq:NLS}). Following the analysis in~\cite[section
3.2]{SC_rings-07}, in the ring region of a standing-ring solution,
i.e., for~$r-\rmax=\mathcal{O}(L)$,
\[
    \left[ \psi_{rr} \right] \sim\frac{[\psi]}{L^2(t)},\qquad
    \left[ \frac{d-1}r\psi_{r} \right]
        \sim \frac{(d-1)[\psi]}{\rmax(T_c)\cdot L(t)}.
\]
Therefore, the~$\frac{d-1}r\psi_{r}$ term in equation~\eqref{eq:NLS}
becomes negligible compared with~$\psi_{rr}$ as~$t\to T_c$.

    Hence, near the singularity, equation~\eqref{eq:NLS} reduces to the
    one-dimensional supercritical NLS~\eqref{eq:1D_SC_NLS}, i.e.,
    \[
        \psi(t,r)\sim \phi\left( t,x=r-\rmax(t) \right),
    \]
    where~$\phi$ is a peak-type solution of the one-dimensional
    NLS~\eqref{eq:1D_SC_NLS}.

Therefore, we predicted in~\cite{SC_rings-07} that the blowup
dynamics of standing ring solutions of the NLS~\eqref{eq:NLS} with
$d>1$ and~$\sigma=2$ is the same as the blowup dynamics of
collapsing peak solutions of the one-dimensional critical NLS
with~$\sigma=2$, as was indeed confirmed numerically
in~\cite{SC_rings-07} and analytically
in~\cite{Raphael-06,Raphael-08}. Similarly, we now predict that the
blowup dynamics of standing-ring solutions of the NLS~\eqref{eq:NLS}
with $d>1$ and~$\sigma>2$ is the same as the blowup dynamics of
collapsing peak solutions of the supercritical one-dimensional
NLS~\eqref{eq:1D_SC_NLS} with the same nonlinearity
exponent~$\sigma$:
\begin{conj}    \label{conj:NLS_ring_equals_peak}
Let~$\psi$ be a singular standing-ring solution of the NLS
equation~\eqref{eq:NLS}
    with~$d>1$ and~$\sigma>2$.  Then,
    \begin{enumerate}
        \item The solution
            is self-similar in the ring region, i.e.,
            $|\psi|\sim|\psi_F|$ for $r-\rmax=\mathcal{O}(L)$,
            where~$\psi_F$ is given by~\eqref{eq:abs_psiF}.
        \item The self-similar profile~$\psi_F$ is given by
            \begin{equation}\label{eq:psiF}
            \psi_F(t,r)= \phi_S\left( t,x=r-\rmax(t) \right),
            \end{equation}
            where~$\phi_S(t,x)$, see~(\ref{eq:psiS}), is the
            asymptotic peak-type profile of the one-dimensional
            NLS~(\ref{eq:1D_SC_NLS}) with the same~$\sigma$.
            In particular, $F=S$ is the admissible solution of
            equation~\eqref{eq:ODE4S} with
            \[
                \LC=\LC_S(\sigma),\qquad S_0=S_0(\sigma).
            \]
        \item The blowup rate of~$\psi$ is a square root, i.e.,
            \begin{equation}\label{eq:NLS_conj_blowup_rate}
                L(t)\sim \LC_S(\sigma)\sqrt{T_c-t}, \qquad t\longrightarrow T_c,
            \end{equation}
            where~$\LC_S(\sigma)$,is the parameter~$\LC=\LC_S$ of the
            self-similar profile~$S$, see~\eqref{eq:unique_parms_ODE4S}.
    \end{enumerate}
\end{conj}

Conjecture~\ref{conj:NLS_ring_equals_peak} implies that the parameter~$\LC$ of
the blowup-rate~\eqref{eq:NLS_conj_blowup_rate} of $\psi$ is equal to the
parameter $\LC$ of the blowup rate~\eqref{eq:intro_square_root_blowuprate} of
$\phi$.
In particular, $\LC$ depends on the nonlinearity exponent~$\sigma$, but
is independent of the dimension~$d$ and of the initial
condition~$\psi_0$.

\subsection{Simulations}\label{sec:simulations_for_NLS_standing} We solve the NLS~\eqref{eq:NLS}
with~$d=2$ and~$\sigma=3$ for the initial
condition
\begin{equation}\label{eq:NLS_d=2_sigma=3_IC}
\psi_0=2e^{-2(r-5)^2},
\end{equation}
and observe that the solution blows-up with a ring profile. In
Figure~\ref{fig:NLS_d=2_sigma=3_standing_ring} we plot the ring
radius\[ \rmax(t)=\arg\max_r|\psi|\]as a function of the focusing
factor~$1/L(t)$, as the solution blows up over 10 orders of
magnitude. Since~$\lim_{t\to T_c} r_{max}(t) = 5.0011$, the ring is
standing and is not shrinking or expanding.
\begin{figure}
    \begin{center}
    \scalebox{.8}{\includegraphics{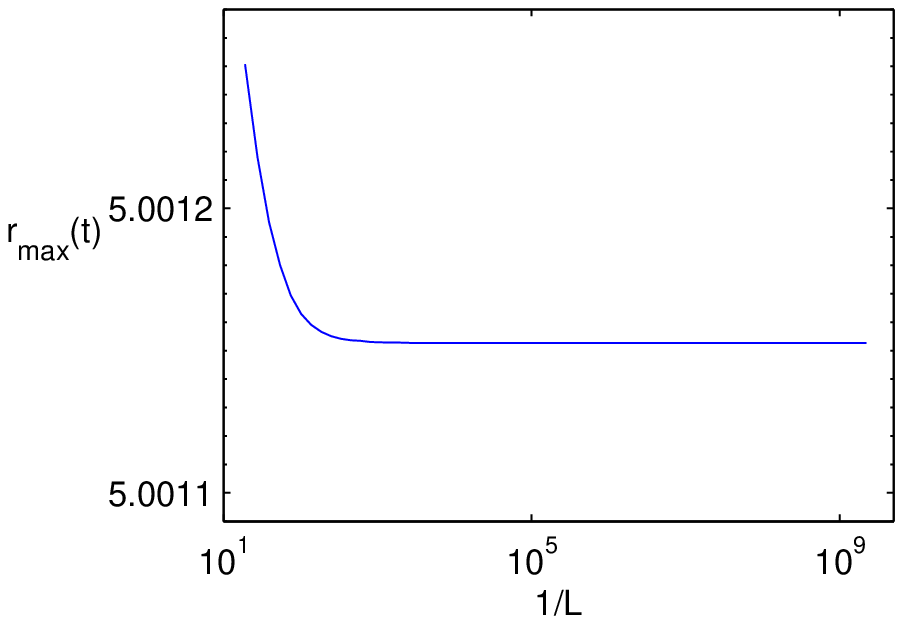}}
    \mycaption{
        Ring radius~$r_{max}(t)$ as a function of the focusing level~$1/L(t)$ for the
        solution of the NLS~(\ref{eq:NLS}) with~$d=2$ and~$\sigma=3$, and
        the initial condition~\eqref{eq:NLS_d=2_sigma=3_IC}.
    \label{fig:NLS_d=2_sigma=3_standing_ring}
    }
    \end{center}
\end{figure}

We now test Conjecture~\ref{conj:NLS_ring_equals_peak} numerically
item by item.
\begin{enumerate}
\item In Figure~\ref{fig:NLS_d=1_sigma=3}A we plot the rescaled solution
\begin{equation}\label{eq:normalization_psi}
    \psi_{\rm rescaled}=
        L^{\frac1\sigma}(t)
            \left| \psi \left(\frac{r-r_{max}(t)}{L(t)}\right)\right|,    \qquad
    L(t) = \left(
        \frac{S_0(\sigma)} {\|\psi(t)\|_\infty}
    \right)^\sigma,
\end{equation}
at~$1/L=10^4$ and~$1/L=10^8$, and observe that the two lines are
indistinguishable. Therefore, we conclude that the standing-ring
solutions blowup with the self-similar $\psi_F$
profile~\eqref{eq:abs_psiF}.
\item To verify that the self-similar blowup profile~$\psi_F$ is, up to a shift in~$r$,
the asymptotic blowup peak-profile~$\phi_S$ of the
one-dimensional NLS~(\ref{eq:1D_SC_NLS}), we superimpose in
Figure~\ref{fig:NLS_d=2_sigma=3}A the self-similar profile of
the solution of the one-dimensional NLS~(\ref{eq:1D_SC_NLS})
from Figure~\ref{fig:NLS_d=1_sigma=3}A and the admissible
solution~$S(x,\sigma=3)$, and observe that, indeed, the four
curves are indistinguishable.
\item Figure~\ref{fig:NLS_d=2_sigma=3}B shows
that
\[
L(t)\sim 1.714\cdot(T_c-t)^{0.5009}.
\]
Therefore, the blowup rate is a square-root or slightly faster.
Figure~\ref{fig:NLS_d=2_sigma=3}C shows that
$
{\displaystyle\lim_{T_c\to t}}LL_t\approx-1.385,
$
indicating that the blowup rate is square-root (with no loglog
correction), i.e.,
\begin{equation}    \label{eq:standing_numeric_BU_rate}
    L(t)\sim \LC^{\rm blowup-rate}_{2D}\sqrt{T_c-t},\qquad
    \LC^{\rm blowup-rate}_{2D}=\sqrt{2\cdot1.385}\approx  1.664.
\end{equation}
Thus, there is an excellent match between the
parameter~$\LC=\LC_S(\sigma=3)\approx  1.664$ of the admissible
S~profile, see~(\ref{eq:ODE4s_parms_sigma=3}), the value
of~$\LC^{\rm blowup-rate}_{2D}\approx  1.664$ extracted from the
blowup rate of the solution of the two-dimensional NLS, and the
value of~$\LC^{\rm blowup-rate}_{1D}\approx  1.664$ extracted from
the blowup rate of the solution of the one-dimensional NLS,
see~\eqref{eq:kappa_blowup_1D}.
\end{enumerate}
\begin{figure}
    \begin{center}
    \scalebox{.8}{\includegraphics{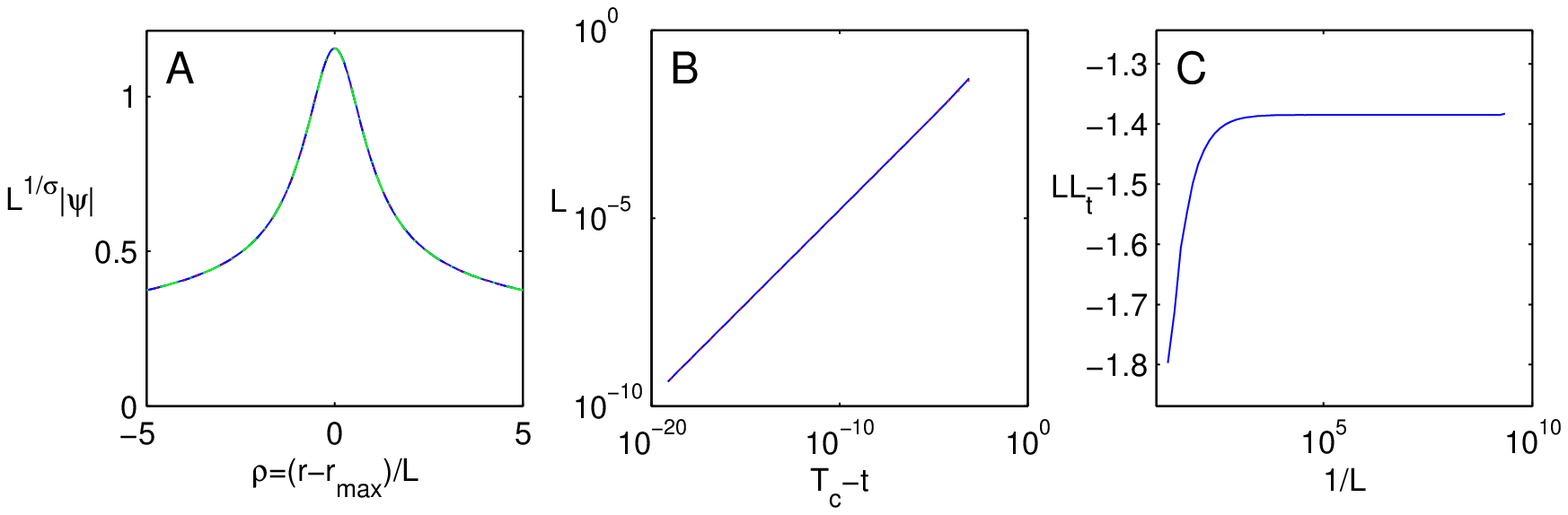}}
    \mycaption{
        NLS solution of Figure~\ref{fig:NLS_d=2_sigma=3_standing_ring}.
        A:~Rescaled solution according to~(\ref{eq:normalization_psi}) at focusing
        levels $1/L=10^4$ (solid) and~$1/L=10^8$ (dashed), dotted curve is the asymptotic
        profile~$S(\xi,\sigma=3)$, and the dashed curve is the rescaled solution of the one-dimensional NLS at~$1/L=10^8$, taken from
        Figure~\ref{fig:NLS_d=1_sigma=3}A.  All four curves are indistinguishable.
        B:~$L$ as a function of~$(T_c-t)$ on a logarithmic scale.
        Dotted curve is the fitted curve
        $1.709( T_c-t)^{0.5007}$.
        C:~$LL_t$ as a function of $1/L$.
        \label{fig:NLS_d=2_sigma=3}
    }
    \end{center}
\end{figure}

\subsection{Robustness of~$\psi_F$ and universality
of~$\kappa$}\label{sec:robustness} The initial
condition~\eqref{eq:NLS_d=2_sigma=3_IC} in
Figures~\ref{fig:NLS_d=2_sigma=3_standing_ring}
and~\ref{fig:NLS_d=2_sigma=3} is different from the asymptotic
profile~$\psi_F$. Since the solution $\psi$ blows up with the
asymptotic profile~$\psi_F$, this indicates that $\psi_F$ is an
attractor. The initial condition~\eqref{eq:NLS_d=2_sigma=3_IC},
however, is already ring-shaped. Therefore, we now show that initial
conditions which are not ring-shaped can also blowup with
the~$\psi_F$ profile.

In~\cite{NGO-08,SG_beams_collapse-06}, we developed a nonlinear
Geometrical Optics (NGO) method which showed that high-power
super-Gaussian initial conditions evolve into a ring profile. To see
this, in Figure~\ref{fig:SG_peak2ring} we solve the
NLS~\eqref{eq:intro_NLS} with~$d=2$ and~$\sigma=3$, and the
super-Gaussian initial condition~$\left.\psi_0(r)=2e^{-r^4}\right.$,
and observe that the NLS solution, indeed, evolves into a ring.
\begin{figure}
    \begin{center}
    \scalebox{.8}{\includegraphics{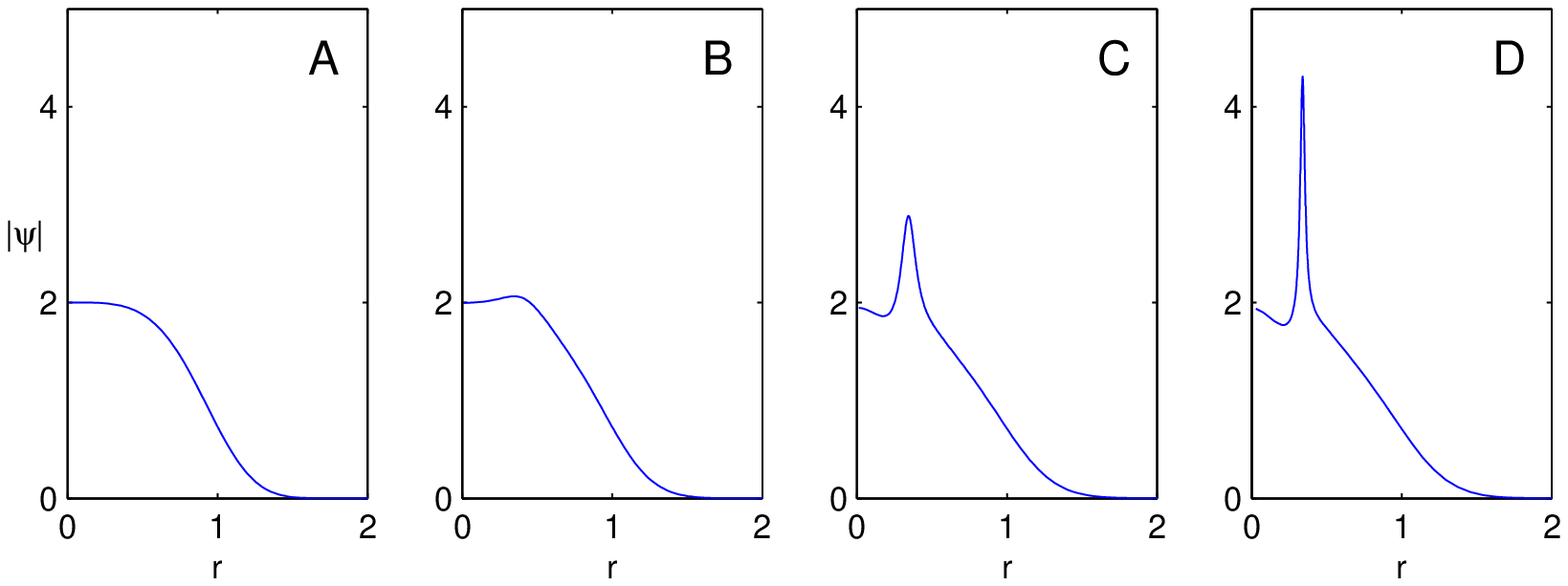}}
    \mycaption{
        Solution of the NLS~(\ref{eq:NLS}) with~$d=2$ and~$\sigma=3$ and
        the initial condition~$\psi_0=2\,e^{-r^4}$ at
        A:~$t=0$. B:~$t=0.0105$. C:~$t=0.0188$. D:~$t=0.0198$.
        \label{fig:SG_peak2ring}
    }
    \end{center}
\end{figure}
Since~$\left.{\displaystyle\lim_{t\to T_c}}r_{max}(t)=0.33\right.$,
see Figure~\ref{fig:NLS_d=2_sigma=3_SG}A, this singular solution is
a standing ring.  Therefore, we see that initial conditions which
are not rings can also blowup with the~$\psi_F$ standing ring
profile.

We now consider the blowup rate of the above solution, since~$\left.
    {\displaystyle \lim_{t\to T_c}}LL_t = -1.384\right.
$, see Figure~\ref{fig:NLS_d=2_sigma=3_SG}B, this implies that \[
    L(t)\sim \LC_{2D}^{\rm blowup-rate} \sqrt{T_c-t},\qquad
    \LC_{2D}^{\rm blowup-rate} \approx \sqrt{2\cdot1.384} = 1.664.
\]
This value of~$\LC_{2D}^{\rm blowup-rate}$ identifies with the one
obtained for the ring-type initial
condition~\eqref{eq:NLS_d=2_sigma=3_IC},
see~\eqref{eq:standing_numeric_BU_rate}.  We thus see that the
parameter $\LC$ of the blowup rate~\eqref{eq:NLS_conj_blowup_rate}
is, indeed, independent of the initial condition.
\begin{figure}
    \begin{center}
    \scalebox{.8}{\includegraphics{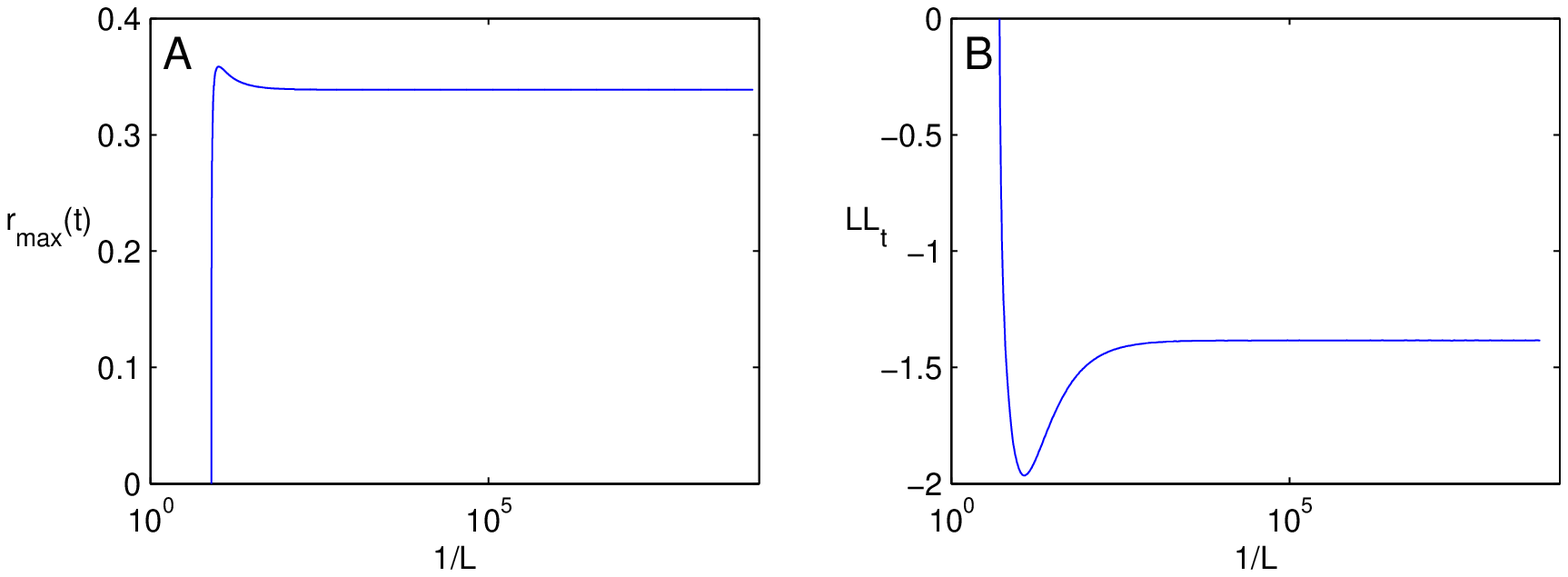}}
    \mycaption{
        NLS solution of Figure~\ref{fig:SG_peak2ring}.
        A:~Location of maximum~$r_{max}(t)$ as a function of the focusing level~$1/L(t)$
        B:~$LL_t$ as a function of the focusing level~$1/L(t)$.
        \label{fig:NLS_d=2_sigma=3_SG}
    }
    \end{center}
\end{figure}

\begin{rmk}
A different type of initial condition that blows-up with
the~$\psi_F$ profile and with the same value of~$\LC_{2D}^{\rm
blowup-rate}$ is given in
Section~\ref{sec:non_standing_simulations}.
\end{rmk}
\subsection{Instability with respect to
symmetry-breaking
perturbations}\label{sec:instablilty_symmetry_breaking} In
Section~\ref{sec:robustness} we saw that the standing-ring
asymptotic profile~$\psi_F$ is an attractor for a large class of
radially-symmetric initial conditions.  In general, NLS solutions
with a ring structure are stable under radial perturbation, but
unstable under symmetry-breaking
perturbations~\cite{Gprofile-05,SC_rings-07,Vortex_rings-08}. We now
show that~$\psi_F$ is also unstable with respect to
symmetry-breaking perturbations. To see that, let us consider the
two-dimensional NLS
\begin{equation}\label{eq:polar_2D_NLS}
i\psi_t(t,r,\theta)+\psi_{rr}+\frac1r\psi_r+\frac1{r^2}\psi_{\theta\theta}+|\psi|^{2\sigma}\psi=0,
\end{equation}
with the initial condition
\[
\psi_0(r,\theta)=f(r)\left(1+\varepsilon h(\theta)\right).
\]
We chose~$f(r)$ so that when~$\varepsilon=0$, the solution blows up
with the~$\psi_F$ profile at~$r=r_{max}$.  We now consider the
case~$0<\varepsilon\ll1$. Since for a standing ring
the~$\frac1r\psi_r$ term becomes negligible compared
with~$\psi_{rr}$, see Section~\ref{sec:theory_for_NLS_standing},
equation~\eqref{eq:polar_2D_NLS} can be approximated in the
ring-peak region~($r\approx r_{max}$) with
\[
i\psi_t+\psi_{rr}+\frac1{r_{max}^2}\psi_{\theta\theta}+|\psi|^{2\sigma}\psi=0.
\]
This is the two-dimensional focusing NLS. Therefore, the solution
will localize at local maximum points in the~$(r,\theta)$ plane,
thereby breaking the radial symmetry.

To see this numerically, we solve the two-dimensional
NLS~\eqref{eq:polar_2D_NLS} with~$\sigma=3$ and with the initial
condition
\begin{equation}\label{eq:non_radial_NLS_IC}
\psi_0(r,\theta)=2e^{-2(r-5)^2}\left[1+\varepsilon^2 e^{-\left(\frac\theta\varepsilon\right)^2}\right],\qquad \varepsilon=\frac1{10},\quad\theta=[-\pi,\pi].
\end{equation}
This initial condition is the standing-ring initial
condition~\eqref{eq:NLS_d=2_sigma=3_IC}, with an~$\mathcal{O}(0.01)$
small bump at~$\left.\theta=0\right.$, see
Figure~\ref{fig:azimuthal_instability}A. As predicted, as the
solution self-focuses, it localizes around the small initial bump
at~$\left.\theta=0\right.$ (see Figure
\ref{fig:azimuthal_instability}B and
\ref{fig:azimuthal_instability}C), resulting in breakup of radial
symmetry.
\begin{figure}
    \begin{center}
    \scalebox{.8}{\includegraphics{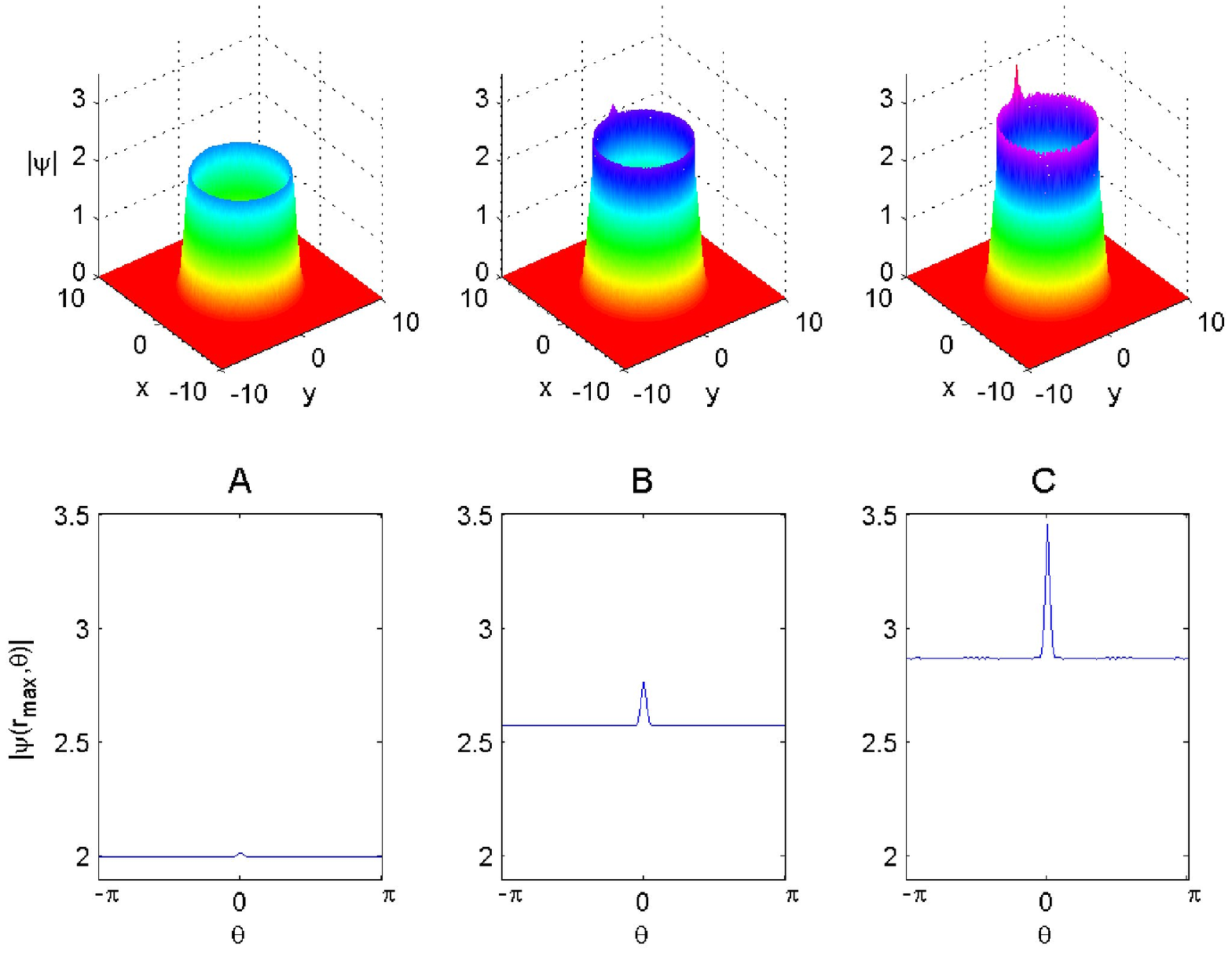}}
    \mycaption{
        Solution of the NLS~\eqref{eq:polar_2D_NLS} with~$\sigma=3$ and with the initial
        condition~\eqref{eq:non_radial_NLS_IC}.
        A:~$t=0$. B:~$t=0.01400$, C:~$t=0.01472$.  Top: Surface plot. Bottom: Amplitude along the ring peak~$|\psi(t,r_{max}(t),\theta)|$ as function of~$\theta$.
        \label{fig:azimuthal_instability}
    }
    \end{center}
\end{figure}
\section{Existence of non-standing ring solutions for~$\sigma>2$ ?}\label{sec:NLS_non-standing_rings}
Lemma~\ref{lem:NLS_sigma_ge_2} does not exclude the possibility that
there exist non-standing rings for~$\sigma>2$. The main reason that
this question arises is as follows. In~\cite{SC_rings-07} we
discovered ring solutions of the supercritical NLS for~$d>1$
and~$\frac2d<\sigma\le2$ of the form~$\psi\sim\psi_Q$,
see~\eqref{eq:psiQ}. Therefore, it is natural to attempt to
extrapolate these results to the regime~$\sigma>2$. Since~$\alpha<0$
when~$\sigma>2$, the ring radius~$r_{max}(t)$ goes to infinity
as~$t\to T_c$, hence~$\psi_Q$ is an {\em expanding-ring} profile for
$\sigma>2$, if such a solution exists. Note that although the ring
is expanding to an infinite radius, the power of the collapsing
part~$\psi_Q$ remains bounded, as
\[
    \norm{\psi_Q}_2^2 =
    \int_{r=\rmax-\rho_c\cdot L(t)}^{\rmax+\rho_c\cdot L(t)}
        |\psi_Q|^2r^{d-1}dr
        \sim r_0^{d-1}\int_{\rho=-\rho_c}^{\rho_c} |Q|^2 d\rho,
        \qquad t\to T_c.
\]
Therefore, these expanding rings, if they exist, do not violate power conservation.

In~\cite{SC_rings-07} we solved the NLS~\eqref{eq:NLS} with
$\sigma=2.1>2$ and~$d=2$ and the ring initial
condition~$\left.\psi_0=\sqrt[4]{3} \sqrt{\sech(2(r - 5))}\right.$.
The solution turned out to be a singular standing ring, rather than
an expanding one. Moreover, the blowup profile was different
from~$\psi_Q$. In retrospect, this NLS solution was a standing-ring
with the~$\psi_F$ profile, see Section~\ref{sec:NLS_standing}.
Nevertheless, this still leaves open the question of whether there
exist expanding-ring~$\psi_Q$ solutions for $\sigma>2$.

\subsection{Analysis}
We now prove that singular ring solutions with the~$\psi_Q$ profile
do not exist for~$\sigma>2$.
\begin{lem} \label{lem:no_expansion}
When~$\sigma>2$, there are no singular NLS solution such
that~$\psi\sim\psi_Q$, see~\eqref{eq:psiQ}, and
\begin{equation}\label{eq:blowup_rate_assumption}
L(t)\sim \LC(T_c-t)^p,\qquad L_t\sim p\,\LC(T_c-t)^{p-1},\qquad L_{tt}\sim
p(p-1)\,\LC(T_c-t)^{p-2}.
\end{equation}
\end{lem}
\begin{proof}
The result shall follow directly from Lemmas~\ref{lem:p_upper_bound}
and~\ref{lem:p_lower_bound}.
\end{proof}

\begin{lem}\label{lem:p_upper_bound}
    Under the assumptions of Lemma~\ref{lem:no_expansion}, $p<1$.
\end{lem}

\begin{proof}
We first recall that, as shown by Merle~\cite{Merle-89}, for every singular
solution $\psi$ of the supercritical NLS
\begin{equation} \label{eq:Merle_bound}
    \int_0^{T_c}(T_c-t)\|\nabla\psi\|_2^2\,dt<\infty.
\end{equation}
To find the limiting behavior of~$\|\nabla \psi(t)\|_2^2$ as~$t\to
T_c$, note that by the conservation of the Hamiltonian
\begin{equation}\label{eq:Hamiltonian_terms_magnitude}
    \|\nabla\psi\|^2_2 \sim
        \frac{1}{\sigma+1}
        \|\psi\|_{2\sigma+2}^{2\sigma+2},\qquad t\to T_c.
\end{equation}
In addition,
\begin{equation}\label{eq:magnitude_of_2sigma} \allowdisplaybreaks
    \|\psi\|_{2\sigma+2}^{2\sigma+2}
        \sim \|\psi_Q\|_{2\sigma+2}^{2\sigma+2}
        = \frac{1}{L^{\frac{2\sigma+2}\sigma}}
            \int |Q(\rho)|^{2\sigma+2}(L\rho+r_0L^\alpha)^{d-1}L\,d\rho
        \sim \frac{r_0^{d-1}}{1+\sigma}\frac{1}{L^2(t)}
            \int |Q|^{2\sigma+2}d\rho,
\end{equation}
where in the last equality we used the value of~$\alpha$ given
by~\eqref{eq:psiQ_alpha}, and, in particular, that~$\alpha<0$.
Therefore,
by~\eqref{eq:blowup_rate_assumption},~\eqref{eq:Hamiltonian_terms_magnitude},~\eqref{eq:magnitude_of_2sigma}
\[
\|\nabla\psi\|_2^2\sim \|\psi\|_{2\sigma+2}^{\sigma+1}\sim \frac{1}{L^2(t)}\sim \frac{1}{(T_c-t)^{2p}},\qquad
t\to T_c.
\]
Hence, the bound~\eqref{eq:Merle_bound} implies that~$p<1$.
\end{proof}

\begin{lem}\label{lem:p_lower_bound}
Under the conditions of Lemma~\ref{lem:no_expansion},~$p>1$.
\end{lem}
\begin{proof}Substitution of~$\psi_Q$, see~\eqref{eq:psiQ}, into the
NLS~\eqref{eq:NLS} gives the following ODE for~$Q$,
\begin{subequations}
\label{eq:QrawODE}
\begin{equation}
\label{eq:QrawODEa}Q_{\rho\rho}(\rho)+\frac{(d-1)L}{L\rho+r_0L^\alpha}Q_\rho-Q+|Q|^{2\sigma}Q-\left[A(t)\rho^2+\alpha r_0 B(t)\rho+\alpha r_0^2 C(t)\right]Q+iD(t)Q=0
,\end{equation} where
\begin{align*}
    & A(t)=\frac14L^3L_{tt}, &&
    B(t)=\frac12L^{2+\alpha}L_{tt}-2(1-\alpha)L^{1+\alpha}L_t^2, &\\
    & C(t)=\frac14L^{1+2\alpha}L_{tt}-(1-\alpha)L^{2\alpha}L_t^2, &&
    D(t)=\frac{\sigma d-2}{2\sigma}\frac{LL_t}{\rho+r_0L^{\alpha-1}}\rho.
\end{align*}
\end{subequations}

Since~$Q$ depends only on~$\rho$, each of the time-dependent terms
of~\eqref{eq:QrawODE} should go to a constant as~$t\to T_c$.  In
particular,~$C(t)$ should go the constant as~$t\to T_c$. Under
assumption~\eqref{eq:blowup_rate_assumption},
\begin{eqnarray}
&C(t)\sim c_C(T_c-t)^{2\alpha p+2p-2},\qquad
&c=\LC^{2+2\alpha}p\left[\left(\alpha-\frac34\right)p-\frac14\right].\label{eq:C_magnitude}
\end{eqnarray}
Since~$\lim_{t\to T_c}L(t)=0$, then~$p>0$,
see~\eqref{eq:blowup_rate_assumption}.  In addition,
for~$\sigma>2$,~$\alpha<0$, see~\eqref{eq:psiQ_alpha}.  This implies
that~$c<0$ and in particular~$c\ne0$. Since~$C(t)$ should goes to a
constant as~$t\to T_c$ then, by~\eqref{eq:C_magnitude},
$\left.2\alpha p+2p-2\ge0\right.$. Therefore,
\[
\frac{1-p}{p}\le \alpha<0.
\]
Hence,~$p>1$.
\end{proof}

\subsection{Simulations}\label{sec:non_standing_simulations}
The result of Lemma~\ref{lem:no_expansion} that expanding-ring
singular solutions with the profile~$\psi_Q$ do not exist, is based
on formal arguments rather than on a rigorous analysis. Therefore,
we now provide a numerical support for this result.  To do that, we
solve the NLS~\eqref{eq:NLS} with~$d=2$ and~$\sigma=3$ and with the
expanding ring profile initial condition
\begin{subequations}\label{eq:IC_expanding}
\begin{equation}
\psi_0=\psi_Q(t=0)=(1+\sigma)^{\frac1{2\sigma}}\mbox{sech}^{\frac1\sigma}(\sigma(r-10))e^{-i\alpha r^2-i(1-\alpha)(r-10)^2},
\end{equation}
where
\begin{equation}
\alpha = \frac{2-3}{3(2-1)}=-\frac13.
\end{equation}
\end{subequations}

If a~$\psi_Q$ solution indeed exists, then~$\psi$ would be a
singular ring solution whose radius goes to infinity. In
Figure~\ref{fig:NLS_d=2_sigma=3_expanding}A we plot the ring
radius~$r_{max}(t)$ as a function of the focusing factor~$1/L$, as
the solution blows up over 10 orders of magnitude. Initially, the
ring radius, indeed, expands from~$r_{max}(0)=10$
to~$r_{max}(t)\approx 12.11$.  This expansion is due to the
defocusing (expanding) phase term~$e^{-i\alpha r^2}$ of the initial
condition. However, the ring stops to expand when~$1/L\approx 20$,
and becomes a singular standing ring with
radius~$r_{max}(T_c)\approx 12.11$. Since the initial condition was
an expanding ring, this simulation provides a strong support to the
result of Lemma~\ref{lem:no_expansion}.

We now consider the blowup rate of the above solution,
Figure~\ref{fig:NLS_d=2_sigma=3_expanding}B shows that~$\left.
    {\displaystyle \lim_{t\to T_c}}LL_t = -1.384,\right.
$ implying that \[
    L(t)\sim \LC_{2D}^{\rm blowup-rate} \sqrt{T_c-t},\qquad
    \LC_{2D}^{\rm blowup-rate} \approx \sqrt{2\cdot1.384} = 1.664.
\]
This value of~$\LC_{2D}^{\rm blowup-rate}$ identifies with the one
obtained for a~$\psi_F$ collapse,
see~\eqref{eq:standing_numeric_BU_rate}.  Therefore, this
simulations provides an additional support to the robustness
of~$\psi_F$ and to the universality of~$\kappa$ (see
Section~\ref{sec:robustness}).
\begin{figure}
    \begin{center}
    \scalebox{.8}{\includegraphics{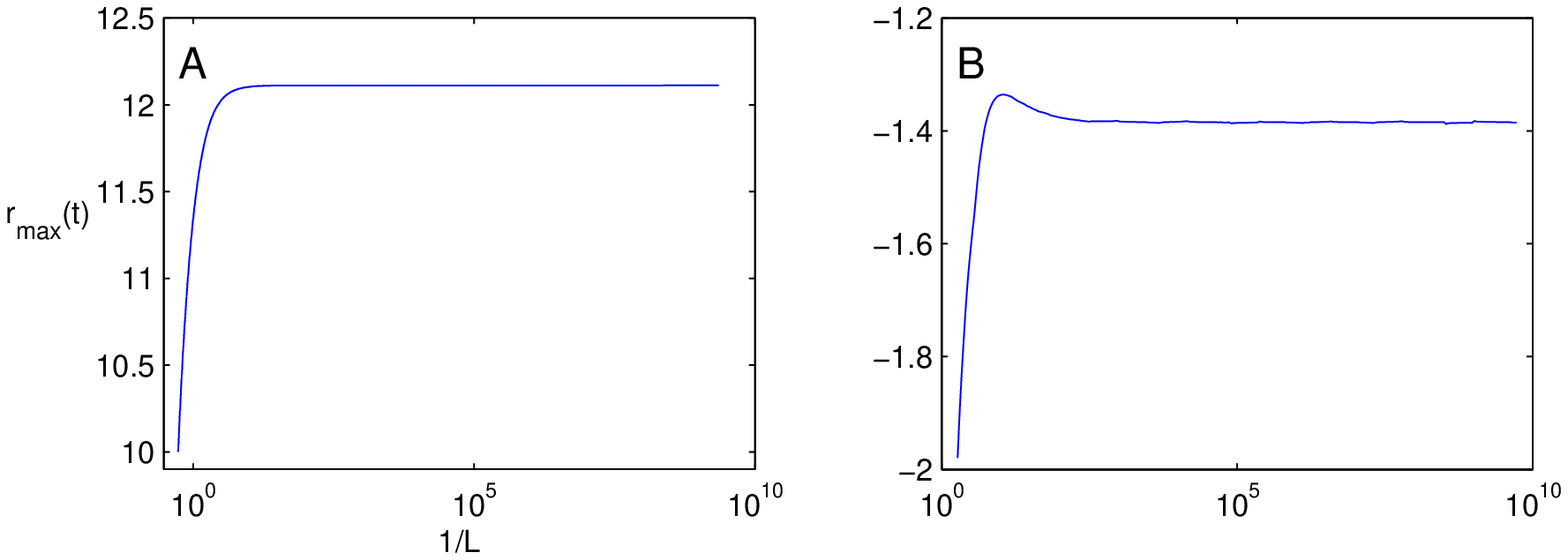}}
    \mycaption{
        Solution of the NLS~(\ref{eq:NLS}) with~$d=2$ and~$\sigma=3$ and
        the initial condition~\eqref{eq:IC_expanding}.
        A: Ring radius~$r_{max}(t)$ as a function of the focusing level~$1/L(t)$.
        B: $LL_t$ as a function of $1/L$.
    \label{fig:NLS_d=2_sigma=3_expanding}
    }
    \end{center}
\end{figure}

\section{Singular standing vortex solutions of the NLS
($\sigma>2$)}\label{sec:vortex} We now consider vortex solutions of
the two-dimensional NLS
\begin{equation}\label{eq:2DNLS}
i\psi_t(t,x,y)+\Delta \psi+|\psi|^{2\sigma}\psi=0,\qquad \psi(0,x,y)=\psi_0(x,y),\qquad \Delta=\partial_{xx}+\partial_{yy},
\end{equation}
i.e., solutions of the form
\begin{equation}\label{eq:vortex_form}
\psi(t,r,\theta)=A(t,r)e^{im\theta},\qquad m\in\mathbb{Z},
\end{equation}
where~$r=\sqrt{x^2+y^2}$ and~$\theta=\tan^{-1}(x/y)$.

In~\cite{Vortex_rings-08} we proved that if the initial condition is
a radially-symmetric vortex, then the solution remains a vortex:
\begin{lem} \label{lem:vortex_IC}
Let~$\psi$ be a solution of the NLS~(\ref{eq:2DNLS}) with the
initial
condition~$\left.\psi_0(r,\theta)=A_0(r)e^{im\theta}\right.$.
Then,~$\psi(t,r,\theta)=A(t,r)e^{im\theta}$, where~$A(t,r)$ is the
solution of
\begin{equation}
iA_t(t,r)+A_{rr}+\frac1rA_r-\frac{m^2}{r^2}A+|A|^{2\sigma}A=0,\qquad
A(0,r)=A_0(r). \label{eq:vortex_NLS}
\end{equation}
\end{lem}
Note that the phase singularity at $r=0$ implies that $A(r=0)=0$.
Hence, all vortex solutions are ring-type solutions.
Specifically, all the singular solutions of~\eqref{eq:vortex_NLS} are ring-typed and
not peak-typed.

In~\cite{Vortex_rings-08} we showed by formal calculations and
numerical simulations that equation~\eqref{eq:2DNLS} admits singular
shrinking-vortex solutions for~$1\le\sigma<2$, and singular
standing-vortex solutions for~$\sigma=2$. Moreover, we showed that
the blowup rate and profile of the standing-vortex solutions is the
same as in the two-dimensional non-vortex case. We now show that
this is also true for~$\sigma>2$, namely, that the analysis
conducted in Section~\ref{sec:theory_for_NLS_standing} for
non-vortex standing-ring collapse, applies also for singular
standing-vortex solutions.

\subsection{Analysis}
\begin{lem} \label{lem:NLS_sigma_ge_2_vortex}
    Let~$\psi$ be a singular standing-ring vortex solution of the NLS~\eqref{eq:2DNLS},
    i.e.,~$\psi\sim\psi_F(t,r)e^{im\theta}$, where
    $\psi_F$ is given by~\eqref{eq:abs_psiF}.
    Then,~$\sigma\ge2$.
    \begin{proof} The proof is identical to the proof of
    Lemma~\ref{lem:NLS_sigma_ge_2}.  Indeed, the proof of Lemma~\ref{lem:NLS_sigma_ge_2} relies only
on~$|\psi|$, hence is not affected by the phase term~$e^{im\theta}$.
\end{proof}
\end{lem}

We now show that the blowup dynamics of standing vortex solutions is
the same as the blowup dynamics of collapsing solutions of the
one-dimensional NLS~\eqref{eq:1D_SC_NLS}.  Indeed, in the ring
region of a standing vortex solution,
\[
    \left[ A_{rr} \right] \sim\frac{[A]}{L^2(t)},\qquad
    \left[ \frac{d-1}rA_{r} \right]
        \sim \frac{(d-1)[A]}{\rmax(T_c)\cdot L(t)},\qquad
    \left[ \frac{m^2}{r^2}A \right]
        \sim \frac{m^2[A]}{\rmax^2(T_c)}.
\]
Therefore, as~$t\to T_c$, both the~$\frac{d-1}rA_{r}$
and~$\frac{m^2}{r^2}A$ terms in equation~\eqref{eq:NLS} become
negligible compared with~$A_{rr}$.

 Hence, as in the non-vortex case, near the singularity,
    equation~\eqref{eq:NLS} reduces to the
    one-dimensional NLS~\eqref{eq:1D_SC_NLS}, i.e.,
    \begin{equation}
        A(t,r)\sim \phi\left( t,x=r-\rmax(t) \right),
    \end{equation}
    where~$\phi$ is a peak-type solution of the one-dimensional
    NLS~\eqref{eq:1D_SC_NLS}.  Therefore, we expect that the blowup dynamics of standing-vortex solutions
of the NLS~\eqref{eq:NLS} with $d=2$ and~$\sigma>2$ to be the same
as the blowup dynamics of collapsing peak solution of the
one-dimensional NLS equation~\eqref{eq:1D_SC_NLS} with the same
nonlinearity exponent~$\sigma$:
\begin{conj}    \label{conj:NLS_ring_equals_peak_vortex}
Let~$\psi(t,r,\theta)=A(t,r)e^{im\theta}$ be a singular
standing-vortex solution of the NLS equation~\eqref{eq:2DNLS}
with~$\sigma>2$.  Then,~$\psi$ blows up with the asymptotic
self-similar profile
\[
\psi\sim e^{im\theta}\cdot\psi_F(t,r),
\]
where~$\psi_F$ is given by~\eqref{eq:psiF}.
In addition, items 1-3 of Conjecture~\ref{conj:NLS_ring_equals_peak}
hold.
\end{conj}

\subsection{Simulations} We solve equation~\eqref{eq:vortex_NLS} with $m=1$ and
$\sigma=3$, with the initial condition
\begin{equation}\label{eq:vortex_simulation_IC}
A_0=2\tanh(4r^2)e^{-2(r-5)^2}.
\end{equation}
In Figure~\ref{fig:NLS_d=2_sigma=3_m=1_standing_ring} we plot the
ring radius~$r_{max}(t)$ as a function of the focusing
factor~$1/L(t)$, as the solution blows up over 10 orders of
magnitude. Since~$\lim_{t\to T_c} r_{max}(t) = 5.0011$, the vortex
is standing.
\begin{figure}
    \begin{center}
    \scalebox{.8}{\includegraphics{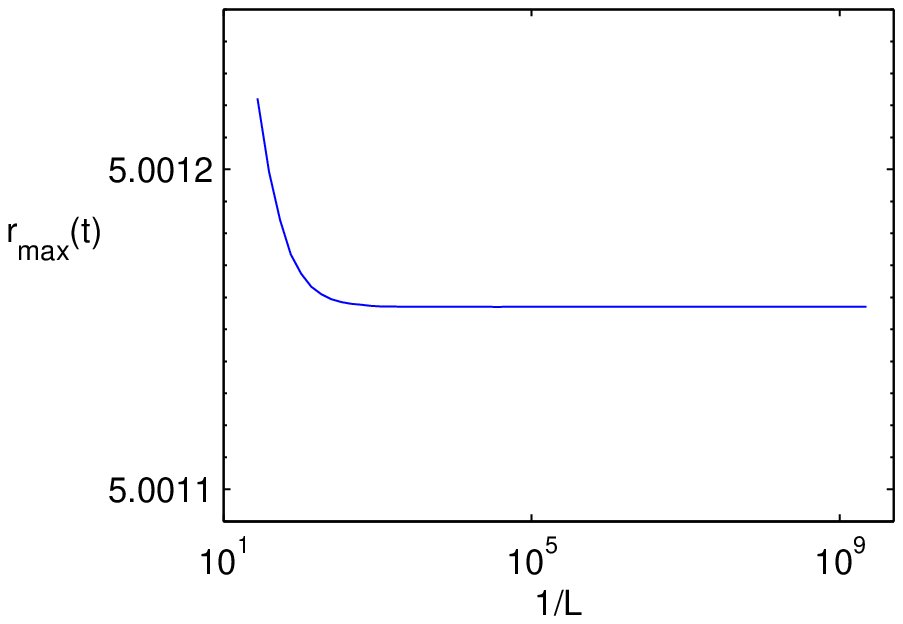}}
    \mycaption{
        Ring radius~$r_{max}(t)$ as a function of the focusing level~$1/L(t)$ for the
        solution of the two-dimensional NLS~(\ref{eq:NLS}) with~$m=1$,~$\sigma=3$ and
        the initial condition~\eqref{eq:vortex_simulation_IC}.
    \label{fig:NLS_d=2_sigma=3_m=1_standing_ring}
    }
    \end{center}
\end{figure}

We now test Conjecture~\ref{conj:NLS_ring_equals_peak_vortex}
numerically item by item.
\begin{enumerate}
\item In Figure~\ref{fig:NLS_d=2_sigma=3_m=1} we plot the rescaled solution, see
        equation~(\ref{eq:normalization_psi}), at~$1/L=10^4$ and~$1/L=10^8$, and
        observe that, indeed, the standing ring solution undergoes a
        self-similar collapse with the profile~\eqref{eq:psiF}.
\item To verify that the self-similar collapse profile is, up to a
        shift in~$r$ and multiplication by~$e^{im\theta}$, the asymptotic collapse profile~$\phi_S$
            of the one-dimensional NLS~(\ref{eq:1D_SC_NLS}), we
superimpose the rescaled solution of the one-dimensional
NLS~(\ref{eq:1D_SC_NLS}) from Figure~\ref{fig:NLS_d=1_sigma=3}A,
as well as the admissible solution~$S(x,\sigma=3)$, onto the
rescaled solutions of Figure~\ref{fig:NLS_d=2_sigma=3_m=1}A and
observe that, indeed, the four curves are indistinguishable.
\item Figure~\ref{fig:NLS_d=2_sigma=3_m=1}B shows
that
\[
L(t)\sim 1.701\cdot(T_c-t)^{0.50068}.
\]
Therefore, the blowup rate is square root or slightly faster.
Figure~\ref{fig:NLS_d=2_sigma=3_m=1}C shows that
\[
\lim_{T_c\to t}LL_t\approx-1.384,
\]
indicating that the blowup rate is square-root, i.e.,
\[
L(t)\sim \LC^{\rm blowup-rate}_{\rm 2D-vortex}\sqrt{T_c-t},\qquad
\LC^{\rm blowup-rate}_{\rm 2D-vortex}=\sqrt{2\cdot1.384}\approx
1.664.
\]
In addition, there is an excellent match between the
parameter~$\LC=\LC_S(\sigma=3)\approx1.664$ of the admissible S
profile, see~(\ref{eq:ODE4s_parms_sigma=3}), and the value
of~$\LC^{\rm blowup-rate}_{\rm 2D-vortex}\approx  1.664$ extracted
from the blowup rate of the two-dimensional vortex solution.
\end{enumerate}
\begin{figure}
    \begin{center}
    \scalebox{.8}{\includegraphics{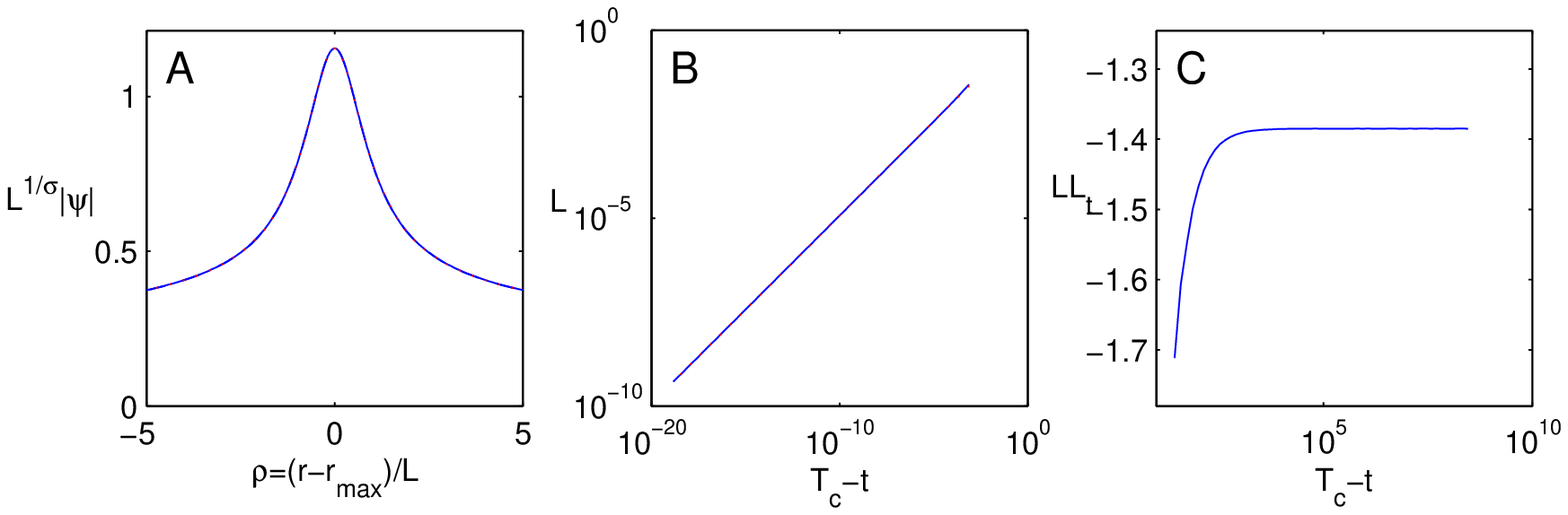}}
    \mycaption{
        Solution of equation~(\ref{eq:vortex_NLS}) with $\sigma=3$, $m=1$, and the
        initial condition~\eqref{eq:vortex_simulation_IC}.
        A:~Rescaled solution according to~(\ref{eq:normalization_psi}) at focusing
        levels $1/L=10^4$ (solid) and~$1/L=10^8$ (dashed), dotted curve is the asymptotic
        profile~$S$, and the dashed curve is the rescaled solution of the
        one-dimensional NLS at~$1/L=10^8$, taken from
        Figure~\ref{fig:NLS_d=1_sigma=3}A.
        All four curves are indistinguishable.
        B:~$L$ as a function of~$(T_c-t)$ on a logarithmic scale.
        Dotted curve is the fitted curve
        $1.71( T_c-t)^{0.5007}$.
        The two curves are indistinguishable.
       C:~$LL_t$ as a function of $1/L$.
       \label{fig:NLS_d=2_sigma=3_m=1}
    }
    \end{center}
\end{figure}

\section{Singular peak-type solutions of the one-dimensional supercritical BNLS}\label{sec:1D_BNLS_collapse}
In Section~\ref{sec:1D_NLS_collapse} we reviewed the theory of
singular peak-type solutions of the one-dimensional NLS. In this
section, we present the analogous findings for the one-dimensional
BNLS. We will make use of these results in the study of singular
standing-ring solutions of the BNLS in
Section~\ref{sec:BNLS_standing}.

\subsection{Analysis}\label{sec:theory_for_1D_BNLS_collapse}
Let us consider the one-dimensional supercritical focusing BNLS
\begin{equation}\label{eq:1D-BNLS}
    i\phi_t(t,x)-\phi_{xxxx}+|\phi|^{2\sigma}\phi=0,\qquad \sigma>4.
\end{equation}
At present, there is no theory for singular peak-type solutions of
equation~\eqref{eq:1D-BNLS}.
A recent numerical study~\cite{Baruch_Fibich_Mandelbaum:2009} suggests that peak-type
singular solutions of the supercritical BNLS~(\ref{eq:1D-BNLS}) collapse with a
self-similar asymptotic profile~$\phi_B$, .i.e.,~$\phi\sim\phi_B$, where
\begin{equation}    \label{eq:phiB}
    \phi_B(t,x)=\frac{1}{L^{2/\sigma}(t)}B(\xi)e^{i\tau},
    \qquad \xi=\frac{x}{L(t)},
    \qquad \tau(t) = \int_{s=0}^{t}\frac{ds}{L^4(s)}.
\end{equation}
The blowup rate~$L(t)$ of these solutions is a quartic root, i.e.,
\begin{equation}    \label{eq:BNLS_4qrt_blowuprate}
    L(t)\sim\LCB\sqrt[4]{T_c-t}, \qquad t\to T_c,
\end{equation} where $\LCB>0$.
In addition, the self-similar profile~$B(\xi)$ is the solution of
\begin{equation}    \label{eq:BNLS-SuperCriticalQEquation}
    \left(
        -1+\frac{i}{2\sigma} \LCB^4
    \right) B(\xi)
    +\frac{i}{4} \LCB^4 \xi B_{\xi}
    - B_{\xi\xi\xi\xi}
    +|B|^{2\sigma}B=0.
 \end{equation}
Note that the NLS analogue of the $B(\xi)$ profile is {\em not} the
$S(\xi)$ profile of equation~\eqref{eq:ODE4S}, but rather
$S(\xi)\cdot e^{i\frac{L_t}{4L}x^2}$. The analogue of the quadratic
phase term in BNLS theory is at present unknown.

In general, symmetric solutions
of~\eqref{eq:BNLS-SuperCriticalQEquation} are complex-valued, and
depend on the parameter~$\LCB$ and on the initial conditions~$B(0)$
and $B^{\prime\prime}(0)$.
We conjecture that, in analogy with the NLS, the following holds:
\begin{conj}    \label{conj:admissible_BNLS}
    \begin{enumerate}
        \item The nonlinear fourth-order ODE~\eqref{eq:BNLS-SuperCriticalQEquation}
            has a unique `admissible solution' with
            $\LCB=\LCB(\sigma)$, $B(0)=B_0(\sigma)$, and
            $B^{\prime\prime}(0)=B_0^{\prime\prime}(\sigma)$.
        \item This admissible solution is the self-similar profile  of the
                asymptotic peak-type blowup profile $\phi_B$, see~\eqref{eq:phiB}.
        \item The value of $\LCB$ of the blowup
            rate~\eqref{eq:BNLS_4qrt_blowuprate} is equal
            to~$\LCB(\sigma)$ of the admissible~$B$
            profile.
    \end{enumerate}
\end{conj}

\subsection{Simulations}    \label{sec:simulations_for_1D_BNLS_collapse}
\begin{figure}
    \begin{center}
    \scalebox{.8}{\includegraphics{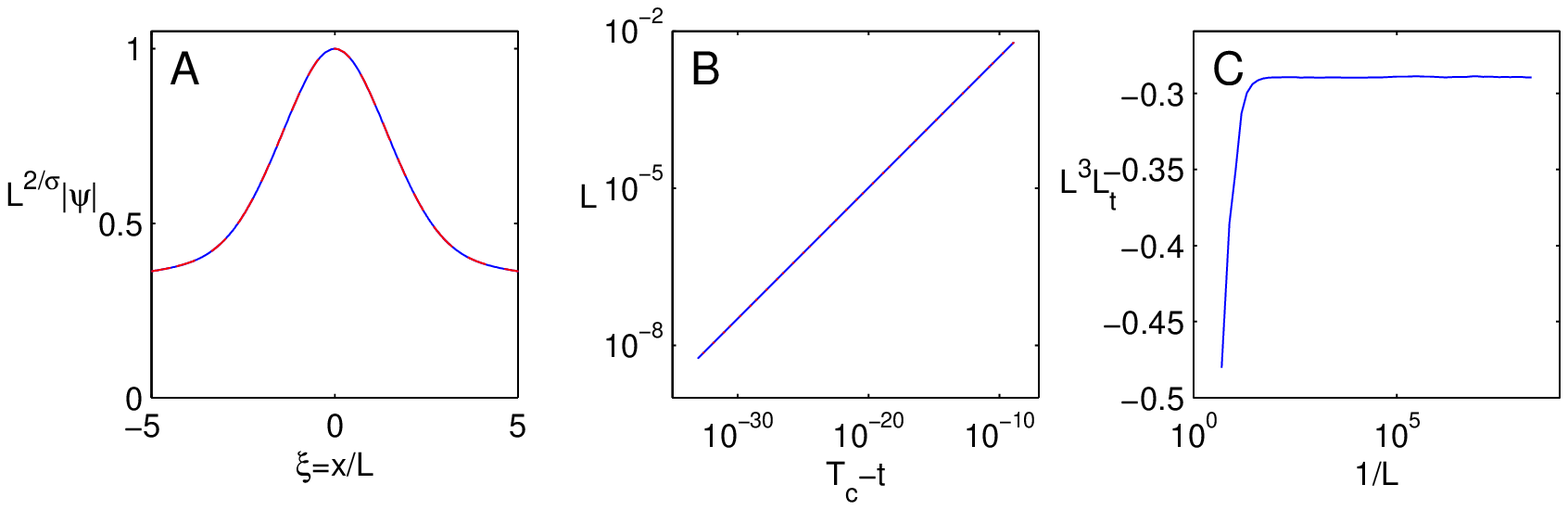}}
    \mycaption{   \label{fig:BNLS_d=1_sigma=6}
        Solution of the one-dimensional BNLS~(\ref{eq:1D-BNLS}) with $\sigma=6$
        and the Gaussian initial condition~\eqref{eq:1D_BNLS_IC}.
        A:~Rescaled solution according to~(\ref{eq:normalization-B}) at focusing
        levels $1/L=10^4$ (solid) and~$1/L=10^8$ (dashed).
        The two curves are indistinguishable.
        B:~$L$ as a function of~$(T_c-t)$ on a logarithmic scale.
        Dotted curve is the fitted curve
        $1.020( T_c-t)^{0.25017}$.
        C:~$L^3L_t$ as a function of $1/L$.
    }
    \end{center}
\end{figure}
We solve the one-dimensional BNLS~(\ref{eq:1D-BNLS}) with~$\sigma=6$
and the Gaussian initial condition
\begin{equation}\label{eq:1D_BNLS_IC}
\phi(t=0,x) = 1.6\cdot e^{-x^2}.
\end{equation}
We first show that the BNLS solution blows up with the self-similar
profile~$\phi_B$, see~\eqref{eq:phiB}. To do that, we rescale the
solution according to
\begin{equation}\label{eq:normalization-B}
    \phi_{\rm rescaled}(t,x)=L^{2/\sigma}(t)
        \phi\left(\frac{x}{L(t)}\right), \qquad
        L(t)= \norm{\phi}_{\infty}^{-\sigma/2}.
\end{equation}
The rescaled profiles at focusing levels of~$1/L=10^4$ and~$1/L=10^8$ are
indistinguishable, see Figure~\ref{fig:BNLS_d=1_sigma=6}A, indicating that the
solution is indeed self-similar while focusing over 4 orders of magnitude.

Next, we consider the blowup rate of the collapsing solution, see
Figure~\ref{fig:BNLS_d=1_sigma=6}B. To do that, we first assume that
\begin{equation}\label{eq:L_assumption-B}
    L(t)\sim \LCB(T_c-t)^p,
\end{equation}
and find the best fitting~$\LCB$ and~$p$. In this case~$\LCB\approx
1.020$ and~$p \approx 0.25017$, indicating that the blowup-rate is
close to a quartic root. To verify that the blowup rate is indeed
$p=1/4$, we compute the limit~${\displaystyle \lim_{t\to T_c}}L^3L_t$.
Note that for the quartic-root blowup rate~\eqref{eq:BNLS_4qrt_blowuprate}
\[
    \lim_{t\to T_c}L^3L_t=-\frac{\LCB^4}4<0,
\]
while for a faster-than-a-quartic root blowup rate~$L^3L_t\to 0$.
Since $ \lim_{T_c\to t}L^3L_t \approx -0.2898$,
see Figure~\ref{fig:BNLS_d=1_sigma=6}C, the blowup rate is a quartic-root (with
no loglog correction), i.e.,
\[
L(t)\sim \LC^{\rm blowup-rate}_{B,1D}\sqrt[4]{T_c-t},\qquad \LC^{\rm
blowup-rate}_{B,1D}\approx\sqrt[4]{4\cdot0.2898} \approx 1.0376.
\]

\section{Singular standing-ring solutions of the supercritical BNLS}
\label{sec:BNLS_standing}

In Section~\ref{sec:NLS_standing} we analyzed singular standing-ring
solutions of the NLS with~$\sigma>2$.  In this section, we derive
the analogous results for the biharmonic NLS with~$\sigma>4$.

\subsection{Analysis} Let us consider singular solutions of the focusing
supercritical BNLS
\begin{equation}    \label{eq:radial-BNLS}
    i\psi_t(t,r) - \Delta^2_r\psi + \left|\psi\right|^{2\sigma}\psi = 0,
    \qquad \sigma d > 4, \qquad d>1,
\end{equation}
where
\begin{equation}    \label{eq:radial_bi_Laplacian}
    \Delta_r^2 =
        -\frac{(d-1)(d-3)}{r^3}\partial_r
        +\frac{(d-1)(d-3)}{r^2}\partial_r^2
        +\frac{2(d-1)}{r}\partial_r^3
        +\partial_r^4
\end{equation}
is the radial biharmonic operator.
The following Lemma, which is the BNLS analogue of
Lemma~\ref{lem:NLS_sigma_ge_2}, shows that standing-ring collapse can only occur
for~$\sigma\geq4$:
\begin{lem} \label{lem:BNLS_sigma_ge_4}
    Let~$\psi$ be a self-similar standing-ring singular solution of the
    BNLS~\eqref{eq:BNLS}, i.e.,~$\psi\sim\psi_B$, where
    \begin{subequations}
    \begin{equation}\label{eq:abs_psiB}
        |\psi_B(t,r)|=\frac{1}{L^{2/\sigma}(t)}|B\left(\rho\right)|,\qquad
            \rho=\frac{r-\rmax(t)}{L(t)},
     \end{equation}
     and
     \begin{equation}
         0< \lim_{t\to T_c}\rmax(t) < \infty.
        \end{equation}
    \end{subequations}
    Then,~$\sigma\ge4$.
    \begin{proof}
        Integration gives $
            \norm{\psi_B}_2^2=\mathcal{O}\left( L^{1-4/\sigma} \right)
        $, and so the proof of Lemma~\ref{lem:NLS_sigma_ge_2} holds for
        $\sigma\ge4$.
    \end{proof}
\end{lem}

\begin{cor} \label{cor:standing_BNLS_strength}
    $\psi_B$ undergoes a strong collapse when~$\sigma=4$, and a weak collapse
    when~$\sigma>4$.
    \begin{proof}
        This follows directly from the proof of Lemma~\ref{lem:BNLS_sigma_ge_4}.
    \end{proof}
\end{cor}

Let us further consider standing-ring solutions of the BNLS~(\ref{eq:radial-BNLS}).
In this case, in the ring region of a standing ring solution, the terms of the
biharmonic operator, see~\eqref{eq:radial_bi_Laplacian}, behave as \[
    \left[ \frac1{r^{4-k}}\partial_r^k\psi  \right]
        = {\cal O}\left( L^{-k}\right)
    ,\qquad k=0,\dots,4.
\]
Therefore, $\Delta_r^2\psi \sim \partial^4_r\psi$.
Hence, near the singularity, equation~\eqref{eq:radial-BNLS} reduces to the
one-dimensional BNLS~\eqref{eq:1D-BNLS}, i.e.,
\begin{equation}
    \psi(t,r)\sim \phi\left( t,x=r-\rmax(t) \right),
\end{equation}
where~$\phi$ is a peak-type solution of the one-dimensional
BNLS~\eqref{eq:1D-BNLS}.
Therefore, we predicted in~\cite{Baruch_Fibich_Mandelbaum:2009} and also
confirmed numerically that the blowup dynamics of standing ring
solutions of the NLS~\eqref{eq:radial-BNLS} with $d>1$
and~$\sigma=4$ is the same as the blowup dynamics of singular peak
solutions of the one-dimensional BNLS~\eqref{eq:radial-BNLS} with~$\sigma=4$.
Similarly, we now predict that the blowup dynamics
of standing ring solutions of the BNLS~\eqref{eq:radial-BNLS} with
$d>1$ and~$\sigma>4$ is the same as the blowup dynamics of
collapsing peak solutions of the one-dimensional
BNLS~\eqref{eq:1D-BNLS} with the same nonlinearity
exponent~$\sigma$:
\begin{conj}    \label{conj:BNLS_ring_equals_peak}
    Let~$\psi$ be a singular standing-ring solution of the
    BNLS~\eqref{eq:radial-BNLS} with~$d>1$ and~$\sigma>4$.
    Then,
    \begin{enumerate}
        \item The solution is self-similar in the ring region, i.e.,
            $\psi\sim \psi_B $ for $r-\rmax=\mathcal{O}(L)$,
            where~$|\psi_B|$ is given by~\eqref{eq:abs_psiB}.
        \item The self-similar profile~$\psi_B$ is given by
            \begin{equation}\label{eq:psiB}
                \psi_B(t,r)= \phi_B\left( t,x=r-\rmax(t) \right),
            \end{equation}
            where~$\phi_B(t,x)$, see~(\ref{eq:phiB}), is the asymptotic profile
            of the one-dimensional BNLS~(\ref{eq:1D-BNLS}) with the
            same~$\sigma$.
        \item The blowup rate is a quartic root, i.e.,
            \begin{equation}\label{eq:BNLS_conj_blowup_rate}
                L(t)\sim \LCB(\sigma)\sqrt[4]{T_c-t}, \qquad t\to T_c,
            \end{equation}
            where $\LCB(\sigma)>0$ is the value of $\LCB$ of the admissible
            $B$ profile.
    \end{enumerate}
\end{conj}
Conjecture~\ref{conj:BNLS_ring_equals_peak} implies, in particular,
that the parameter~$\LCB$ of the blowup-rate of~$\psi$,
see~\eqref{eq:BNLS_conj_blowup_rate}, is the same as the parameter $\LCB$ of the
blowup rate of~$\phi$, see~\eqref{eq:BNLS_4qrt_blowuprate}.
This value depends on the nonlinearity exponent~$\sigma$, but is independent of
the dimension~$d$ and of the initial condition~$\psi_0$.

\subsection{Simulations}

We solve the BNLS~\eqref{eq:radial-BNLS} with~$d=2$ and~$\sigma=6$
with the initial condition~$\left.\psi_0(r) = 1.6\cdot
e^{-(r-5)^2}\right.$. In
Figure~\ref{fig:BNLS_d=2_sigma=6_standing_ring} we plot the ring
radius~$\rmax(t)$ as a function of the focusing factor~$1/L(t)$, as
the solution blows up over 8 orders of magnitude. Since~$\lim_{t\to
T_c} \rmax(t) = 4.992$, the ring is standing and is not shrinking or
expanding. Note that the initial condition is different from the
asymptotic profile~$\psi_B$, suggesting that {\em standing-ring
collapse is (radially) stable}.
\begin{figure}
    \begin{center}
    \scalebox{.8}{\includegraphics{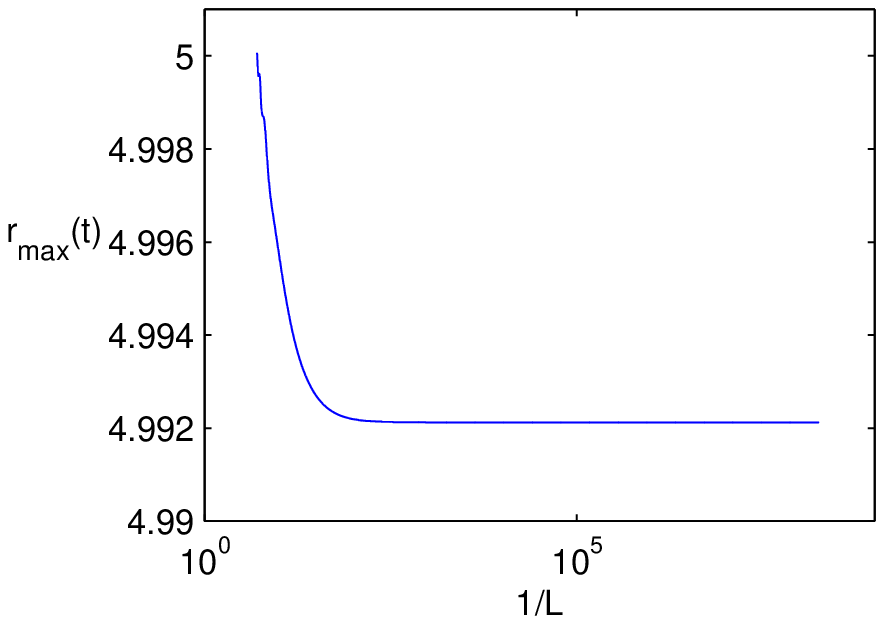}}
    \mycaption{
        Ring radius~$\rmax(t)$ as a function of the focusing level~$1/L(t)$ for the
        solution of the BNLS~(\ref{eq:BNLS}) with $d=2$, $\sigma=6$ and
        the initial condition~$\psi_0(r) = 1.6\cdot e^{-(r-5)^2}$.
        \label{fig:BNLS_d=2_sigma=6_standing_ring}
    }
    \end{center}
\end{figure}

We next test each item of
Conjecture~\ref{conj:BNLS_ring_equals_peak} numerically:
\begin{enumerate}
    \item In Figure~\ref{fig:BNLS_d=2_sigma=6}A, we plot the rescaled solution
        \begin{equation}\label{eq:normalization_psi_B}
            \psi_{\rm rescaled} = L^{2/\sigma}(t)
                \psi\left(\frac{r-\rmax(t)}{L(t)}\right),
            \qquad L(t)=\norm{\psi}_{\infty}^{-\sigma/2},
            \qquad \rmax(t)=\arg\max_r|\psi|,
        \end{equation}
        at $1/L=10^4$ and $1/L=10^8$.
        The two curves are indistinguishable, showing that standing rings
        undergo a self-similar collapse with the self-similar~$\psi_B$
        profile~\eqref{eq:phiB}.
    \item To verify that the self-similar collapse profile~$\psi_B$ is, up to a
        shift in~$r$, equal to the asymptotic collapse
        profile~$\phi_B$ of the one-dimensional
        BNLS~(\ref{eq:1D-BNLS}), we superimpose the rescaled
        solution of the one-dimensional BNLS~(\ref{eq:1D-BNLS})
        from Figure~\ref{fig:BNLS_d=1_sigma=6}A, onto the
        rescaled solutions of
        Figure~\ref{fig:BNLS_d=2_sigma=6}A, and observe that,
        indeed, the curves are indistinguishable.
    \item Figure~\ref{fig:BNLS_d=2_sigma=6}B shows that
        \[
            L(t)\sim 1.020(T_c-t)^{0.25017}.
        \]
        Therefore, the blowup rate is quartic root or slightly
        faster. Figure~\ref{fig:BNLS_d=2_sigma=6}C shows that $
        {\displaystyle \lim_{T_c\to t}}L^3L_t\approx -0.2894, $
        indicating that the blowup rate is quartic-root (with no
        loglog correction), i.e., \[ L(t)\sim \LC_{B, 2D}^{\rm
            blowup-rate}\sqrt[4]{T_c-t},\qquad
                        \LC_{B, 2D}^{\rm blowup-rate} \approx \sqrt[4]{4\cdot0.2894} \approx
                        1.0373.
        \]
        As predicted, there is an excellent match between the value
        of~$\LC_{B,2D}^{\rm blowup-rate}\approx 1.0373$ extracted from the
        two-dimensional BNLS solution, and the value
        of~$\LC^{\rm blowup-rate}_{B,1D}\approx1.0376$ extracted from the
        one-dimensional BNLS solution,
        see Section~\ref{sec:simulations_for_1D_BNLS_collapse}.
\end{enumerate}

\begin{figure}
    \begin{center}
    \scalebox{.8}{\includegraphics{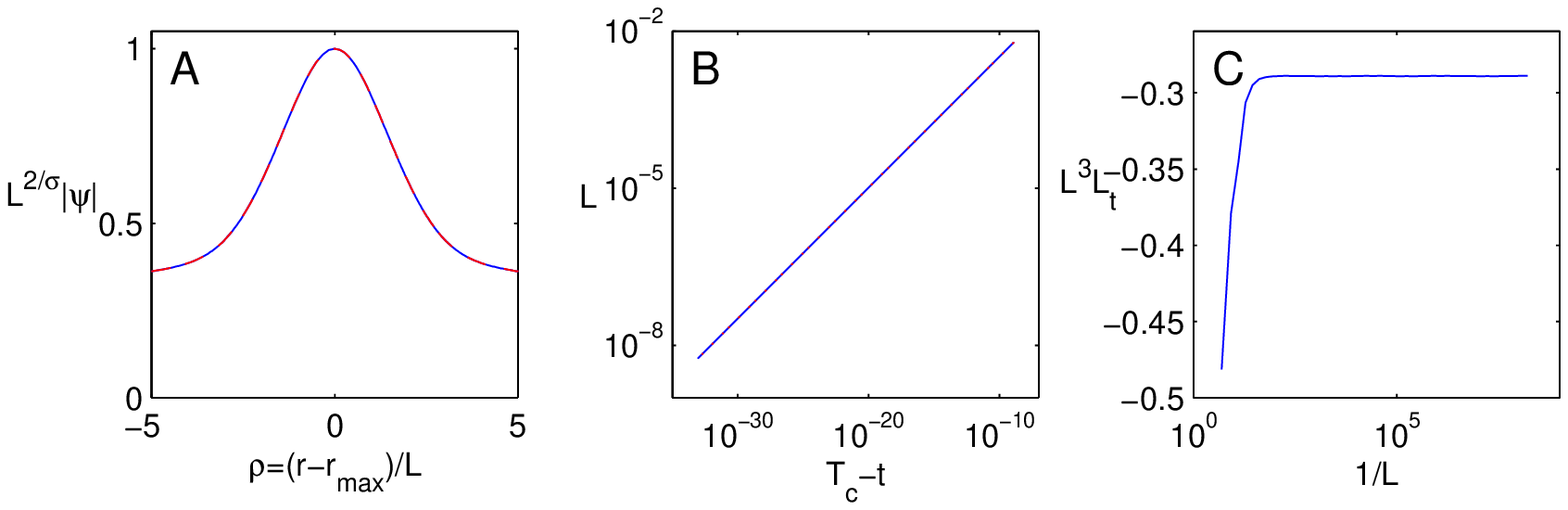}}
    \mycaption{
        \label{fig:BNLS_d=2_sigma=6}
        Solution of the two-dimensional BNLS~(\ref{eq:BNLS}) with $\sigma=6$ and
        the initial condition~$\psi_0(x)=1.6\cdot e^{-(r-5)^2}$.
        A:~Rescaled solution according to~(\ref{eq:normalization-B}) at focusing
        levels $1/L=10^4$ (solid) and~$1/L=10^8$ (dashed).
        The dashed curve is the rescaled solution of the one-dimensional BNLS
        at~$1/L=10^8$, taken from Figure~\ref{fig:BNLS_d=1_sigma=6}.
        All three curves are indistinguishable.
        B:~$L$ as a function of~$(T_c-t)$ on a logarithmic scale.
        Dotted curve is the fitted curve $1.020( T_c-t)^{0.2502}$.
        The two curves are indistinguishable.
        C:~$L^3L_t$ as a function of $1/L$.
    }
    \end{center}
\end{figure}

\section{Singular standing-ring solutions of the nonlinear heat
equation}\label{sec:NLH} The $d$-dimensional radially-symmetric
nonlinear heat equation (NLHE)
\begin{equation}    \label{eq:NLHE}
u_t(t,r)-\Delta u-|u|^{2\sigma}u=0,\qquad \sigma>0,\quad d>1,
\end{equation}
where~$u$ is real and~$\Delta =\partial_{rr}+\frac{d-1}r\partial_r$,
admits singular solutions for any~$\sigma>0$~\cite{GigaKohn_85}. To
the best of our knowledge, until now all known singular solutions
of~\eqref{eq:NLHE} collapsed at a point (or at a finite number of
points~\cite{Merle_Heat_1992}). We now show that the NLHE admits
also singular standing-ring solutions.  The blowup profile and
blowup rate of these solutions are the same as those of singular
peak-type solutions of the one-dimensional NLHE with the
same~$\sigma$.

\subsection{Peak-type solutions of the one-dimensional NLHE (review)}
one-dimensional NLHE\footnote{Throughout this paper, we denote the
solution of the one-dimensional NLHE by~$v$, and its spatial
variable by~$x$.}
\begin{equation}    \label{eq:NLHE_1D}
v_t(t,x)-v_{xx}-|v|^{2\sigma}v=0,\qquad \sigma>0,
\end{equation}
admits singular solutions that collapse with the self-similar
peak-type profile
\begin{subequations}\label{eq:NLHE_1D_prof}
\begin{equation}\label{eq:NLHE_1D_prof_u}
v_S(t,x)=\frac{1}{\lambda^{\frac1\sigma}(t)}\frac1{\left(1+\xi^2\right)^{\frac1{2\sigma}}},\qquad \xi=\frac{x}{L(t)},
\end{equation}
where
\begin{equation}\label{eq:NLHE_1D_prof_L}
\lambda(t)=\sqrt{2\sigma(T_c-t)},\qquad L(t)=\sqrt{2\left(2+\frac1\sigma\right)(T_c-t)|\ln(T_c-t)|},
\end{equation}
\end{subequations}
see~\cite{NonlinearHeatBlowup-08}. Note that unlike the
one-dimensional NLS, the one-dimensional NLHE admits singular
solutions for any~$\sigma>0$.
\subsection{Analysis}
Let us consider singular standing-ring solutions of the
NLHE~\eqref{eq:NLHE}. Since~$\Delta u \sim u_{rr}$ in the ring
region, near the singularity equation~\eqref{eq:NLHE} reduces to the
one-dimensional NLHE~\eqref{eq:NLHE_1D}. Therefore, we conjecture
that singular standing-ring solutions of the NLHE~\eqref{eq:NLHE}
exist for any~$\sigma>0$, and that the blowup profile and blowup
rate of these solutions are the same as those of singular peak-type
solutions of the one-dimensional NLHE with the same~$\sigma$.
\begin{conj}    \label{conj:heat_ring_equals_peak}
    Let~$u(t,r)$ be a singular standing-ring solution of the
    NLHE~\eqref{eq:NLHE}.
    Then, the solution is self-similar in the ring region,
i.e.,~$u\sim u_S$ for $r-\rmax=\mathcal{O}(L)$, where
            \begin{equation}\label{eq:uS}
                u_S(t,r)= v_S\left( t,x=r-\rmax(t) \right),
            \end{equation}
and~$v_S$ is given by equation~\eqref{eq:NLHE_1D_prof}.
\end{conj}
 \subsection{Simulations}
We solve the NLHE~\eqref{eq:NLHE} with~$d=2$ and~$\sigma=3$ and the
initial condition
\begin{equation}\label{eq:heat_IC}
u_0(r)=2e^{-2(r-5)^2}.
\end{equation}
The solution blows up with a ring profile, see
Figure~\ref{fig:heat_initial_focusing}.
\begin{figure}
    \begin{center}
    \scalebox{.8}{\includegraphics{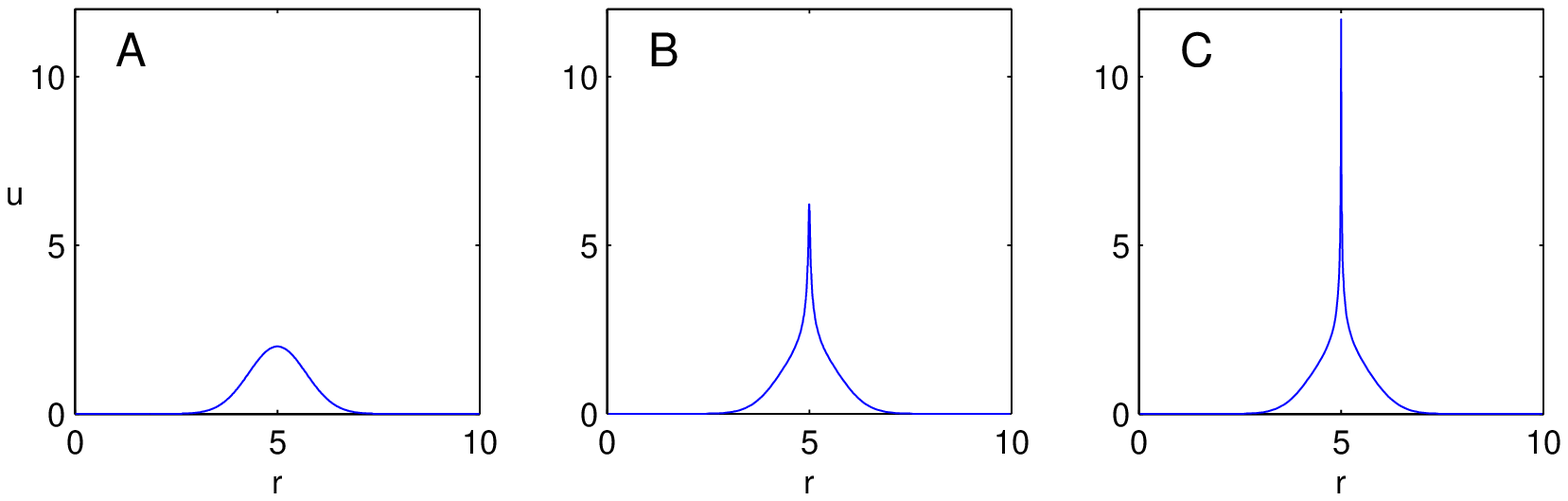}}
    \mycaption{
        Solution of the NLHE~\eqref{eq:NLHE} with $d=2$,~$\sigma=3$ and
        the initial condition~\eqref{eq:heat_IC}. A:~$t=0$. B:~$t=0.002683$. C:~$t=0.002686$.
        \label{fig:heat_initial_focusing}}
    \end{center}
\end{figure}
Since~$\lim_{t\to T_c} r_{max}(t) = 4.9994>0$, see
Figure~\ref{fig:heat_d=2_sigma=3}A, the ring is standing.
\begin{figure}
    \begin{center}
    \scalebox{.8}{\includegraphics{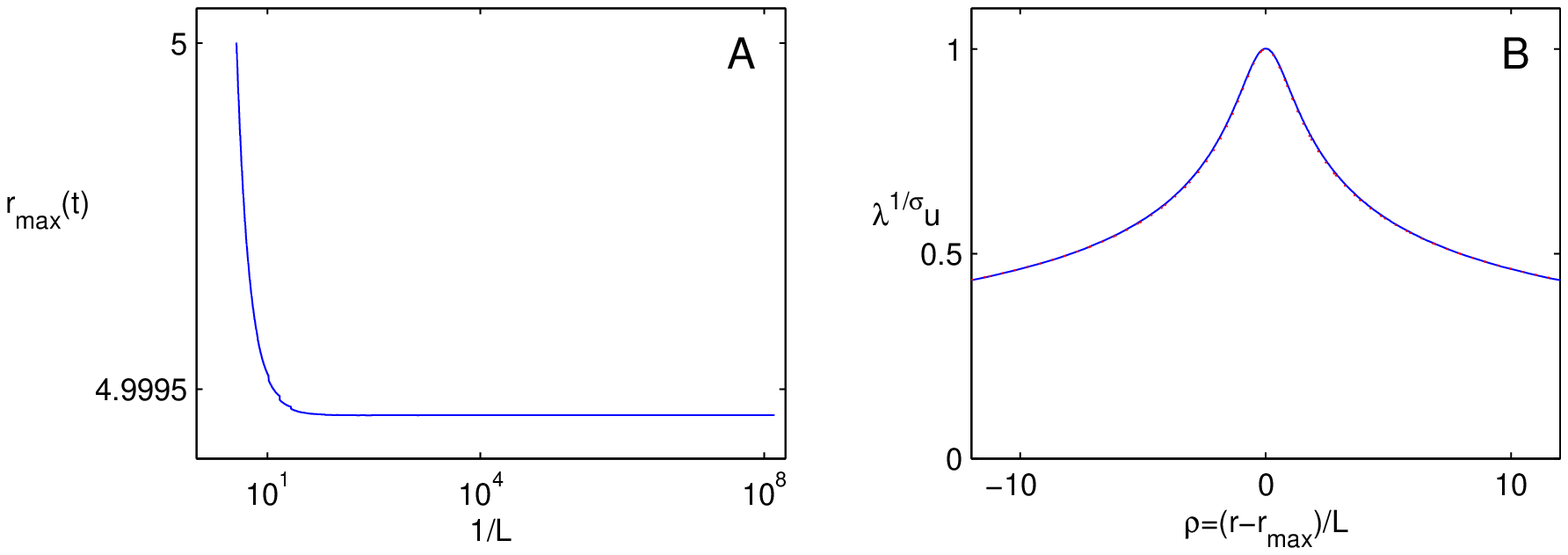}}
    \mycaption{
        Solution of Figure~\ref{fig:heat_initial_focusing}.
        A: Ring radius~$\rmax(t)$ as a function of~$1/L(t)$.
        B: Solution~$u$ at~$1/L=10^{8}$  (solid curve).
           Dashed curve is the asymptotic~$u_S$ profile~\eqref{eq:rescaledUs}.
            Both curves are rescaled according to~\eqref{eq:heat_normalization}. The two
           curves are indistinguishable.
        \label{fig:heat_d=2_sigma=3}
    }
    \end{center}
\end{figure}

In order to show that the solution blows up with the
self-similar~$u_S$ profile~\eqref{eq:uS}, we rescale the solution
according to
\begin{equation}\label{eq:heat_normalization}
u_{\rm rescaled}(t,r)=\lambda^{\frac1\sigma}(t)u\left(\frac{r-r_{max}}{L(t)}\right),
\end{equation}
where~$\lambda(t)$ and~$L(t)$ are given by~\eqref{eq:NLHE_1D_prof_L}
and~$T_c$ is extracted from the numerical simulation. The rescaled
profile at~$1/L=10^8$ is in perfect agreement with the
rescaled~$u_S$ profile
\begin{equation}\label{eq:rescaledUs}
u_{S,{\rm rescaled}}=\frac1{\left(1+\rho^2\right)^{\frac1{2\sigma}}},
\end{equation}
see Figure~\ref{fig:heat_d=2_sigma=3}B. Therefore, the numerical
results provide a strong support to
Conjecture~\ref{conj:heat_ring_equals_peak}.
\section{Singular standing-ring solutions of the nonlinear biharmonic heat equation}\label{sec:BNLHE}
The $d$-dimensional radially-symmetric biharmonic nonlinear heat
(BNLHE) equation
\begin{equation}    \label{eq:BNLHE}
    u_t(t,r)+\Delta^2 u-|u|^{2\sigma}u=0,\qquad \sigma>0,\quad d>1,
\end{equation}
where~$u$ is real and~$\Delta^2$ is the radial biharmonic
operator~\eqref{eq:radial_bi_Laplacian}, admits singular solutions
for any~$\sigma>0$~\cite{Budd-Williams-Galaktionov_2004}. To the
best of our knowledge, all known singular solutions of the
BNLHE~\eqref{eq:BNLHE} collapse at a point. We now show that the
BNLHE admits singular standing-ring solutions. The blowup profile
and blowup rate of these solutions are the same as those of singular
peak-type solutions of the one-dimensional BNLHE with the
same~$\sigma$.

\subsection{Peak-type solutions of the one-dimensional BNLHE (review)}

The one-dimensional BNLHE equation
\begin{equation}    \label{eq:BNLHE_1D}
    v_t(t,x)+v_{xxxx}-|v|^{2\sigma}v=0,\qquad \sigma>0,
\end{equation}
admits singular peak-type solutions which collapse with the
self-similar peak-type profile
\begin{equation}\label{eq:BNLHE_1D_prof}
    v_B(t,x)= \frac{1}{L^{2/\sigma}(t)} B(\xi),\qquad
    L(t) = \kappa_{\scriptscriptstyle \text{BH}} \sqrt[4]{T_c-t},\qquad
    \xi=\frac{x}{L(t)},
\end{equation}
see~\cite{Budd-Williams-Galaktionov_2004}. The self-similar
profile~$B(\xi)$ is not known explicitly. Unlike the one-dimensional
BNLS, the one-dimensional BNLHE admits singular solutions for
any~$\sigma>0$.

\subsection{Analysis}
Let us consider singular standing-ring solutions of the BNLHE~\eqref{eq:BNLHE}.
Since~$\Delta^2 u \sim u_{rrrr}$ in the ring region, near the singularity the
BNLHE~\eqref{eq:BNLHE} reduces to the one-dimensional BNLHE~\eqref{eq:BNLHE_1D}.
We therefore conjecture that standing-ring solutions of
equation~\eqref{eq:BNLHE} exist, and that their blowup profile and blowup rate
is the same as those of singular peak solutions of
equation~\eqref{eq:BNLHE_1D}:
\begin{conj}    \label{conj:Bheat_ring_equals_peak}
    Let~$u(t,r)$ be a singular standing-ring solution of the
    BNLHE~\eqref{eq:BNLHE}.
    Then, the solution is self-similar in the ring region, i.e.,$u\sim u_B$ for $r-\rmax=\mathcal{O}(L)$, where
            \begin{equation}\label{eq:uBH}
                u_B(t,r)= v_B\left( t,x=r-\rmax(t) \right),
            \end{equation}
and~$v_B$ is given by equation~\eqref{eq:BNLHE_1D_prof}.
\end{conj}

\subsection{Simulations}
We solve the BNLHE~\eqref{eq:BNLHE} with~$d=2$ and~$\sigma=3$ and the initial
condition~\eqref{eq:heat_IC}.
The solution blows up with a standing-ring profile, see
Figure~\ref{fig:BiHeat}A.
\begin{figure}
    \begin{center}
    \scalebox{.8}{\includegraphics{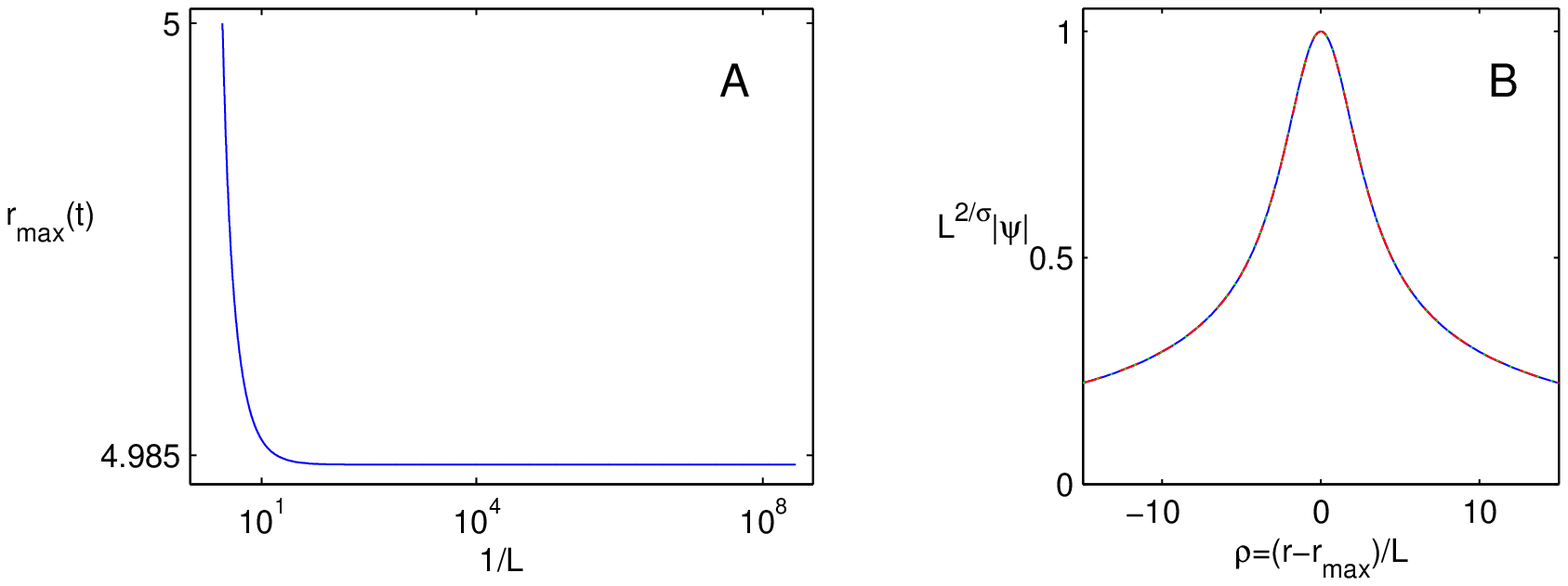}}
    \mycaption{
        Standing ring solution of the BNLHE~\eqref{eq:BNLHE} with initial condition~\eqref{eq:heat_IC}.
        A: ring radius~$\rmax(t)$ as a function of~$1/L(t)$
        B: solution~$u$ rescaled according
        to~\eqref{eq:normalization_psi_B}, at focusing factors of $1/L=10^{4}$ (solid curve)
        and $1/L=10^{8}$ (dashed curve).
        The rescaled peak-type solution of the one-dimensional NLH~\eqref{eq:BNLHE_1D} is given by the dotted
        curve.
        All three curves are indistinguishable.
        \label{fig:BiHeat}
    }
    \end{center}
\end{figure}
In Figure~\ref{fig:BiHeat}B we plot the solution, rescaled according
to~\eqref{eq:normalization_psi_B}, at focusing levels of
$1/L=10^{4}$ and $1/L=10^{8}$. The two curves are indistinguishable,
indicating that the solution is indeed self-similar. Next, we want
to show that the self-similar blowup profile is given by~$B(\xi)$,
the self-similar profile of peak-type solutions of the
one-dimensional NLHE, see~\eqref{eq:BNLHE_1D_prof}. To do that, we
compute the solution of the one-dimensional
BNLHE~\eqref{eq:BNLHE_1D} with~$\sigma=3$ and~$u(t=0,x)=2e^{-x^2}$,
and superimpose its profile at~$1/L=10^8$ in
Figure~\ref{fig:BiHeat}B. The three rescaled solutions are
indistinguishable, indicating that standing-ring solutions of the
BNLHE blowup with the self-similar profile of peak-type solutions of
the one-dimensional BNLHE.
In addition, we have from the numerical simulations that~$\lim_{t\to
T_c}L^3L_t\approx-1.2108$ when~$d=2$,
where~$L(t):=\|u\|_\infty^{-\sigma/2}$, and~$\lim_{t\to
T_c}L^3L_t\approx-1.2108$ when~$d=1$,
where~$L(t):=\|v\|_\infty^{-\sigma/2}$. Therefore, the blowup rate
of the standing-ring solution of the two-dimensional BNLHE is equal,
up to 5 significant digits, to the blowup rate of the singular
peak-type solution of the one-dimensional BNLHE, and is given by
\[
    L(t)\sim \kappa_{\scriptscriptstyle \text{BH}}\sqrt[4]{T_c-t},\qquad
        \kappa_{\scriptscriptstyle \text{BH}}\approx\sqrt[4]{4\cdot 1.2108}\approx1.4835.
\]
Thus, the numerical results provide a strong support for
Conjecture~\ref{conj:Bheat_ring_equals_peak}.

\section{Numerical Methods}\label{sec:numerical_methods}
\subsection{Admissible solutions of~\eqref{eq:ODE4S}}
The admissible solution~$S$ of~\eqref{eq:ODE4S} was calculated using
the shooting method of Budd, Chen and
Russel~\cite[Section~3.1]{Budd-99}. In this method, one searches in
the two-parameter space $(S(0),f_c)$ for the parameters such that
the solution of~\eqref{eq:ODE4S} satisfies the admissible solution
condition
\[
\lim_{\xi\to\infty}F(\xi;f_c,S(0))=0,\qquad F(\xi;f_c,S(0))=\left|\xi
S^\prime(\xi)+\left(\frac1\sigma+\frac2{f_c^2}i-i\frac{f_c^2}4
\xi^2\right)S(\xi)\right|^2.
\]
\subsection{Solution of the NLS, BNLS, NLHE and BNLHE}
In this study, we computed singular solutions of the
NLS~\eqref{eq:radial-NLS}, the BNLS~\eqref{eq:radial-BNLS}, the
NLHE~\eqref{eq:NLHE} and the BNLHE~\eqref{eq:BNLHE}. Theses
solutions become highly-localized, so that the spatial
scale-difference between the singular region $r-\rmax = {\cal O}(L)$
and the exterior regions can be as large as ${\cal O}(1/L)\sim
10^{10}$. In order to resolve the solution at both the singular and
non-singular regions, we use an adaptive grid.

We generate the adaptive grids using the {\em Static Grid
Redistribution} (SGR) method, which was first introduced by Ren and
Wang~\cite{Ren-00}, and was later simplified and improved by Gavish
and Ditkowsky~\cite{SGR-08}. Using this approach, the solution is
allowed to propagate (self-focus) until it starts to become
under-resolved.  At this stage, a new grid, with the same number of
grid-points, is generated using De'Boors `equidistribution
principle', see~\cite{Ren-00,SGR-08} for details.

The method in~\cite{SGR-08} also allows control of the portion of
grid points that migrate into the singular region, preventing
under-resolution at the exterior regions.
In~\cite{Baruch_Fibich_Mandelbaum:2009}, we further extended the
approach to prevent under-resolution in the transition region~${\cal
O}(L)\ll r-\rmax \ll O(1)$.

On the sequence of grids, the equations are solved using a
Predictor-Corrector Crank-Nicholson scheme.
\subsubsection*{Acknowledgments.}
This research was partially supported by the Israel Science
Foundation (ISF grant No. 123/08). The research of Nir Gavish was
also partially supported by the Israel Ministry of Science Culture
and Sports.


\end{document}

%% file: NLS_ring_class.pstex_t
\begin{picture}(0,0)%
\includegraphics{NLS_ring_class.pstex}%
\end{picture}%
\setlength{\unitlength}{3947sp}%
\begingroup\makeatletter\ifx\SetFigFontNFSS\undefined%
\gdef\SetFigFontNFSS#1#2#3#4#5{%
  \reset@font\fontsize{#1}{#2pt}%
  \fontfamily{#3}\fontseries{#4}\fontshape{#5}%
  \selectfont}%
\fi\endgroup%
\begin{picture}(5499,5134)(-11,-4208)
\put(3151,-4036){\makebox(0,0)[lb]{\smash{{\SetFigFontNFSS{29}{34.8}{\rmdefault}{\mddefault}{\updefault}{\color[rgb]{0,0,0}$\mathbf\sigma$}%
}}}}
\put(1276,-2611){\makebox(0,0)[lb]{\smash{{\SetFigFontNFSS{17}{20.4}{\rmdefault}{\mddefault}{\updefault}{\color[rgb]{1,1,1}\underline{\it A}}%
}}}}
\put(2701,-1336){\makebox(0,0)[b]{\smash{{\SetFigFontNFSS{17}{20.4}{\rmdefault}{\mddefault}{\updefault}{\color[rgb]{0,0,0}\underline{\it C}}%
}}}}
\put(2701,-1936){\makebox(0,0)[b]{\smash{{\SetFigFontNFSS{17}{20.4}{\rmdefault}{\mddefault}{\updefault}{\color[rgb]{0,0,0}$0<\alpha<1$}%
}}}}
\put(4801,-1636){\makebox(0,0)[b]{\smash{{\SetFigFontNFSS{17}{20.4}{\rmdefault}{\mddefault}{\updefault}{\color[rgb]{0,0,0}\underline{\it E}}%
}}}}
\put(1201,-61){\makebox(0,0)[b]{\smash{{\SetFigFontNFSS{17}{20.4}{\rmdefault}{\mddefault}{\updefault}{\color[rgb]{0,0,0}$\alpha=1$ (equal-rate)}%
}}}}
\put(1201,539){\makebox(0,0)[b]{\smash{{\SetFigFontNFSS{17}{20.4}{\rmdefault}{\mddefault}{\updefault}{\color[rgb]{0,0,0}\underline{\it B}}%
}}}}
\put(4201,-61){\makebox(0,0)[b]{\smash{{\SetFigFontNFSS{17}{20.4}{\rmdefault}{\mddefault}{\updefault}{\color[rgb]{0,0,0}$\alpha=0$}%
}}}}
\put(4201,539){\makebox(0,0)[b]{\smash{{\SetFigFontNFSS{17}{20.4}{\rmdefault}{\mddefault}{\updefault}{\color[rgb]{0,0,0}\underline{\it D}}%
}}}}
\put(1201,239){\makebox(0,0)[b]{\smash{{\SetFigFontNFSS{17}{20.4}{\rmdefault}{\mddefault}{\updefault}{\color[rgb]{0,0,0}$\sigma d =2$ (critical)}%
}}}}
\put(2701,-1636){\makebox(0,0)[b]{\smash{{\SetFigFontNFSS{17}{20.4}{\rmdefault}{\mddefault}{\updefault}{\color[rgb]{0,0,0}$2/d<\sigma <2$}%
}}}}
\put(1276,-2911){\makebox(0,0)[lb]{\smash{{\SetFigFontNFSS{17}{20.4}{\rmdefault}{\mddefault}{\updefault}{\color[rgb]{1,1,1}$\sigma d<2$}%
}}}}
\put(4201,239){\makebox(0,0)[b]{\smash{{\SetFigFontNFSS{17}{20.4}{\rmdefault}{\mddefault}{\updefault}{\color[rgb]{0,0,0}$\sigma=2$}%
}}}}
\put(4801,-1936){\makebox(0,0)[b]{\smash{{\SetFigFontNFSS{17}{20.4}{\rmdefault}{\mddefault}{\updefault}{\color[rgb]{0,0,0}$\sigma>2$}%
}}}}
\end{picture}%

%% file: BNLS_ring_class.pstex_t
\begin{picture}(0,0)%
\includegraphics{BNLS_ring_class.pstex}%
\end{picture}%
\setlength{\unitlength}{3947sp}%
\begingroup\makeatletter\ifx\SetFigFontNFSS\undefined%
\gdef\SetFigFontNFSS#1#2#3#4#5{%
  \reset@font\fontsize{#1}{#2pt}%
  \fontfamily{#3}\fontseries{#4}\fontshape{#5}%
  \selectfont}%
\fi\endgroup%
\begin{picture}(5499,5209)(-11,-4208)
\put(3151,-4036){\makebox(0,0)[lb]{\smash{{\SetFigFontNFSS{29}{34.8}{\rmdefault}{\mddefault}{\updefault}{\color[rgb]{0,0,0}$\mathbf\sigma$}%
}}}}
\put(1276,-2611){\makebox(0,0)[lb]{\smash{{\SetFigFontNFSS{17}{20.4}{\rmdefault}{\mddefault}{\updefault}{\color[rgb]{1,1,1}\underline{\it A}}%
}}}}
\put(1276,-2911){\makebox(0,0)[lb]{\smash{{\SetFigFontNFSS{17}{20.4}{\rmdefault}{\mddefault}{\updefault}{\color[rgb]{1,1,1}$\sigma d<4$}%
}}}}
\put(2701,-1336){\makebox(0,0)[b]{\smash{{\SetFigFontNFSS{17}{20.4}{\rmdefault}{\mddefault}{\updefault}{\color[rgb]{0,0,0}\underline{\it C}}%
}}}}
\put(4801,-1636){\makebox(0,0)[b]{\smash{{\SetFigFontNFSS{17}{20.4}{\rmdefault}{\mddefault}{\updefault}{\color[rgb]{0,0,0}\underline{\it E}}%
}}}}
\put(4801,-1936){\makebox(0,0)[b]{\smash{{\SetFigFontNFSS{17}{20.4}{\rmdefault}{\mddefault}{\updefault}{\color[rgb]{0,0,0}$\sigma>4$}%
}}}}
\put(1201,239){\makebox(0,0)[b]{\smash{{\SetFigFontNFSS{17}{20.4}{\rmdefault}{\mddefault}{\updefault}{\color[rgb]{0,0,0}$\sigma d =4$ (critical)}%
}}}}
\put(1201,539){\makebox(0,0)[b]{\smash{{\SetFigFontNFSS{17}{20.4}{\rmdefault}{\mddefault}{\updefault}{\color[rgb]{0,0,0}\underline{\it B}}%
}}}}
\put(4201,239){\makebox(0,0)[b]{\smash{{\SetFigFontNFSS{17}{20.4}{\rmdefault}{\mddefault}{\updefault}{\color[rgb]{0,0,0}$\sigma=4$}%
}}}}
\put(4201,539){\makebox(0,0)[b]{\smash{{\SetFigFontNFSS{17}{20.4}{\rmdefault}{\mddefault}{\updefault}{\color[rgb]{0,0,0}\underline{\it D}}%
}}}}
\put(1201,-61){\makebox(0,0)[b]{\smash{{\SetFigFontNFSS{17}{20.4}{\rmdefault}{\mddefault}{\updefault}{\color[rgb]{0,0,0}$\alpha_B=1$ (equal-rate)}%
}}}}
\put(4201,-61){\makebox(0,0)[b]{\smash{{\SetFigFontNFSS{17}{20.4}{\rmdefault}{\mddefault}{\updefault}{\color[rgb]{0,0,0}$\alpha_B=0$}%
}}}}
\put(2701,-1636){\makebox(0,0)[b]{\smash{{\SetFigFontNFSS{17}{20.4}{\rmdefault}{\mddefault}{\updefault}{\color[rgb]{0,0,0}$4/d<\sigma <4$}%
}}}}
\put(2701,-1936){\makebox(0,0)[b]{\smash{{\SetFigFontNFSS{17}{20.4}{\rmdefault}{\mddefault}{\updefault}{\color[rgb]{0,0,0}$0<\alpha_B<1$}%
}}}}
\end{picture}%